\documentclass[final,reqno]{amsart}
\usepackage{amsmath,amsthm,amsfonts,amscd,amssymb}
\usepackage[hmargin={28mm,28mm},vmargin={30mm,35mm}]{geometry}
\usepackage{array}
\usepackage{enumerate} 
\usepackage{graphicx}
\usepackage{enumitem}
\usepackage{tikz}
\usepackage{xspace}
\usepackage[ocgcolorlinks,allcolors={blue},breaklinks]{hyperref}
\usepackage{xcolor}
 \usepackage[normalem]{ulem}
 \newcounter{corr}
 \definecolor{violet}{rgb}{0.580,0.,0.827}
 \newcommand{\corr}[3]{\typeout{Warning : a correction remains in page
 \thepage}
 				\stepcounter{corr}        
 				{\color{blue}\ifmmode\text{\,\sout{\ensuremath{#1}}\,}\else\sout{#1}\fi}
         {\color{red}#2}
         {\color{violet} #3}}

\numberwithin{equation}{section}
\newtheorem{theorem}{Theorem}
\newtheorem{lemma}[theorem]{Lemma}

\theoremstyle{remark}
\newtheorem{remark}[theorem]{Remark}
\theoremstyle{definition}

\newtheorem{definition}[theorem]{Definition}

\newcommand{\Real}{\mathbb{R}}

\newcommand{\norme}[1]{\left\Vert #1\right\Vert}

\newcommand{\ug}{\boldsymbol{u}}
\newcommand{\vg}{\boldsymbol{v}}

\newcommand{\DIV}{\text{div }}
\newcommand{\wg}{\boldsymbol{w}}

\newcommand{\normal}{\mathbf{n}}
\newcommand{\BS}[1]{\boldsymbol{#1}}

\title[Periodic solutions to a fluid--structure system]{{Existence of time-periodic strong solutions to a fluid--structure system}}

\author{Jean-J\'er\^ome Casanova}
\address{Institut de Math\'ematiques de Toulouse, UMR 5219, Universit\'e Paul Sabatier Toulouse III.}
\email{jean-jerome.casanova@math.univ-toulouse.fr}

\begin{document}
\begin{abstract}
We study a nonlinear coupled fluid--structure system modelling the blood flow through arteries. The fluid is described by the incompressible Navier--Stokes equations in a 2D rectangular domain where the upper part depends on a structure satisfying a damped Euler--Bernoulli beam equation. The system is driven by time-periodic source terms on the inflow and outflow boundaries. We prove the existence of time-periodic strong solutions for this problem under smallness assumptions for the source terms.
\end{abstract}
\maketitle

\section{Introduction}
In this paper we are interested in the existence of time-periodic solutions for a fluid--structure system involving the incompressible Navier--Stokes equations coupled with a damped Euler--Bernoulli beam equation located on a part of the fluid domain boundary. This system can be used to model the blood flow through human arteries and serves as a benchmark problem for FSI solvers in hemodynamics. When the system is driven by periodic source terms, related for example to the periodic heartbeat, we expect a periodic response of the system. In this article, we prove the existence of time-periodic solutions for the fluid--structure system subject to small periodic impulses on the inflow and outflow boundaries. The study of this fluid--structure model in a periodic framework seems to be new. 
\\

For $L>0$ consider the domain $\Omega$ in $\mathbb{R}^{2}$ defined by $\Omega=(0,L)\times(0,1)$. The different components of the boundary $\partial\Omega$ are denoted by: $\Gamma_{i}=\{0\}\times (0,1)$, $\Gamma_{o}=\{L\}\times (0,1)$, $\Gamma_{b}=(0,L)\times\{0\}$, $\Gamma_{s}=(0,L)\times\{1\}$ and $\Gamma_{d}=\Gamma_{s}\cup\Gamma_{i}\cup\Gamma_{b}$. Let $T>0$ be a period of the system, the domain of the fluid at the time $0\leq t \leq T$ is denoted by $\Omega_{\eta(t)}$ and depends on the displacement of the beam $\eta :\Gamma_s\times(0,T)\mapsto (-1,+\infty)$. More precisely
\[
\begin{array}{rr}
\begin{aligned}
\Omega_{\eta(t)}&=\{(x,y)\in \mathbb{R}^2\mid x\in (0,L),\, 0<y<1+\eta(x,t)\},\\
\Gamma_{\eta(t)}&=\{(x,y)\in \mathbb{R}^2\mid x\in (0,L),\, y=1+\eta(x,t)\}.\\
\end{aligned}
\end{array}
\]
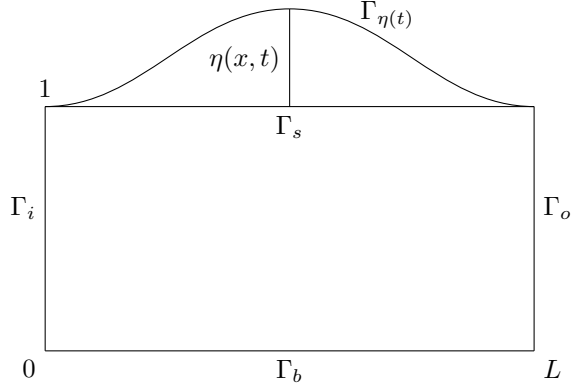
\begin{figure}[h!]
\centering
\begin{tikzpicture}[scale=0.65]
\draw (5,0)node [below] {$\Gamma_b$};
\draw (0,0)node [below left] {$0$};
\draw (10,0)node [below right] {$L$};
\draw (0,5)node [above] {$1$};
\draw (10,2.5)node [above right] {$\Gamma_{o}$};
\draw (0,2.5) node [above left] {$\Gamma_{i}$};
\draw (5,5) node [below] {$\Gamma_s$};
\draw (7,6.4) node [above] {$\Gamma_{\eta(t)}$};
\draw (0,0) -- (10,0);
\draw (10,0) -- (10,5);
\draw (0,5) -- (0,0);
\draw (5,5) -- (5,7);
\draw (0,5) -- (10,5);
\draw (5,5.5)node [above left] {$\eta(x,t)$};
\draw (0,5) .. controls (2,5) and (3,7) .. (5,7);
\draw (5,7) .. controls (7,7) and (8,5) .. (10,5);
\end{tikzpicture}
\caption{Fluid--structure system.} 
\end{figure}
For space-time domain we use the notations
\begin{align*}
&\Sigma^{s}_{T}=\Gamma_{s}\times(0,T),\,\Sigma^{i}_{T}=\Gamma_{i}\times(0,T),\,\Sigma^{o}_{T}=\Gamma_{o}\times(0,T),\,\Sigma^{b}_{T}=\Gamma_{b}\times(0,T),\\
&\Sigma^{d}_{T}=\Gamma_{d}\times(0,T),\,\Sigma_{T}^{\eta}=\underset{t\in(0,T)}{\bigcup}\Gamma_{\eta(t)}\times\{t\},\,\,Q_{T}^{\eta}=\underset{t\in(0,T)}{\bigcup}\Omega_{\eta(t)}\times\{t\}.\\
\end{align*}
Consider the $T$-periodic fluid--structure system
\begin{equation}\label{chap2-existence1}
\begin{aligned}
&\ug_{t} + (\ug\cdot\nabla)\ug-\DIV\sigma(\ug,p)=0,\,\,\,\DIV\ug=0\,\text{ in }Q^{\eta}_{T},\\
&\ug=\eta_{t}\textbf{e}_{2}\,\text{ on }\Sigma^{\eta}_{T},\\
&\ug=\boldsymbol{\omega}_{1}\,\text{ on }\Sigma^{i}_{T},\\
&u_{2}=0\,\text{ and }\,p+(1/2)\vert\ug\vert^{2}=\omega_{2}\,\text{ on }\Sigma^{o}_{T},\\
&\ug=0\,\text{ on }\Sigma^{b}_{T},\,\,\,\ug(0)=\ug(T)\,\text{ in }\Omega_{\eta(0)},\\
&\eta_{tt}-\beta\eta_{xx}-\gamma\eta_{txx}+\alpha\eta_{xxxx}
=-J_{\eta(t)}\textbf{e}_{2}\cdot \sigma(\ug,p)_{\vert\Gamma_{\eta(t)}}\normal_{\eta(t)}\,\text{ on }\Sigma^{s}_{T},\\
&\eta=0\,\text{ and }\,\eta_{x}=0\,\text{ on }\{0,L\}\times(0,T),\\
&\eta(0)=\eta(T)\,\text{ and }\,\eta_{t}(0)=\eta_{t}(T)\,\text{ in }\Gamma_{s},
\end{aligned}
\end{equation}
where $\ug=(u_{1},u_{2})$ is the fluid velocity, $p$ the pressure, $\eta$ the displacement of the beam and
\begin{align*}
\sigma(\ug,p)&=-pI+\nu(\nabla\ug + (\nabla\ug)^{T}),\\
\normal_{\eta(t)}&=J_{\eta(t)}^{-1}\begin{pmatrix}
-\eta_{x}(x,t)\\
1
\end{pmatrix},
\end{align*}
with $J_{\eta(t)}=\sqrt{1+\eta_{x}^{2}}$. The constants $\beta\geq 0$, $\gamma>0$, $\alpha>0$ are parameters relative to the structure and $\nu>0$ is the constant viscosity of the fluid. The periodic source terms $(\boldsymbol{\omega}_{1},\omega_{2})$ play the role of a `pulsation' for the system and can model the heartbeat.

The fluid--structure system \eqref{chap2-existence1} has been investigated with different conditions on the inflow and outflow boundaries:
\begin{enumerate}
\item[(DBC)] homogeneous Dirichlet boundary conditions.
\item[(PBC)] periodic boundary conditions.
\item[(PrBC)] pressure boundary conditions.
\end{enumerate}
For (DBC), the existence of strong solutions is proved in \cite{MR2027753,MR2765696,MR2745779}. The first result, stated in \cite{MR2027753}, is the existence of local-in-time strong solutions for small data. This result is then improved in \cite{MR2745779}, where the stabilization process directly implies the existence of strong solutions, on an arbitrary time interval $[0,T]$ with $T>0$, for small data. Finally, in \cite{MR2765696}, the existence of strong solutions for small data and of local-in-time strong solutions without smallness assumptions on the initial data is proved. As specified in \cite{2017arXiv170706382C}, the strategy developed in \cite{MR2765696} works for zero (or small) initial beam displacement. This difficulty, purely nonlinear, was solved in \cite{2017arXiv170706382C} and more recently in \cite{grandmont:hal-01567661}.

For (PBC), the existence of global strong solutions without smallness assumptions on the initial data is proved in \cite{MR3466847}. For a wide range of beam equations, depending on the positivity of the coefficients $\beta,\gamma,\alpha$, the existence of local-in-time strong solutions without smallness assumptions is proved in \cite{MR3466847,grandmont:hal-01567661}.

The third case (PrBC) is introduced in \cite{MR3017292} where the existence of weak solutions is proved. We investigated in \cite{2017arXiv170706382C} the existence of local-in-time strong solutions without smallness assumptions on the initial data, which includes non-small initial beam displacement, and the existence of strong solution on $[0,T]$ with $T>0$ for small data.

Here we are interested in the existence of time-periodic strong solutions. The term `strong solutions' is related to the spacial regularity of the solution, which is typically, for the fluid, $\BS{H}^{2}$. In the semigroup terminology of evolution equations, the solutions considered in \cite{MR2027753,2017arXiv170706382C,MR3466847,grandmont:hal-01567661,MR2765696,MR2745779} correspond to strict solutions in $L^{2}$ (see Definition \ref{chap2-defi.sol} in the appendix). Motivated by the stabilization of \eqref{chap2-existence1} in a neighbourhood of a periodic solution, we prove the existence of a time-periodic strict solution in $\mathcal{C}^{0}$ for \eqref{chap2-existence1} with H\"older regularity in time. Our result can be directly adapted for the boundary conditions (DBC)--(PBC)--(PrBC). The Dirichlet boundary condition on the inflow is motivated, once again, by stabilization purpose.

Let us describe the general strategy to construct a periodic solution for \eqref{chap2-existence1}. First, we perform a change of variables mapping the moving domain $\Omega_{\eta(t)}$ into the fixed domain $\Omega$. We then linearize and we rewrite the coupled system as an abstract evolution equation driven by an unbounded operator $(\mathcal{A},\mathcal{D}(\mathcal{A}))$ in Section \ref{section2}. We prove that $(\mathcal{A},\mathcal{D}(\mathcal{A}))$ is the infinitesimal generator of an analytic semigroup and that its resolvent is compact. At this stage we use the abstract results developed in the appendix to ensure the existence of a time-periodic solution for the linear system. Finally, we study the nonlinear system in Section \ref{chap2-nonlinear} with a fixed point argument in the space of periodic functions. The main theorem of this paper, where the notation $\sharp$ denotes time-periodic functions, can be formulated as follows.
\begin{theorem}\label{main-theorem-intro}
Fix $\theta\in(0,1)$ and $T>0$. There exists $R>0$ such that, for all $T$-periodic source terms
\[(\boldsymbol{\omega}_{1},\omega_{2})\in\left(\mathcal{C}^{\theta}_{\sharp}([0,T];\mathbf{H}^{3/2}_{0}(\Gamma_{i}))\cap \mathcal{C}^{1+\theta}_{\sharp}([0,T];\mathbf{H}^{-1/2}(\Gamma_{i}))\right)\times\mathcal{C}^{\theta}_{\sharp}([0,T];H^{1/2}(\Gamma_{o})),
\]
satisfying
\[\norme{\boldsymbol{\omega}_{1}}_{\mathcal{C}^{\theta}_{\sharp}([0,T];\mathbf{H}^{3/2}_{0}(\Gamma_{i}))\cap \mathcal{C}^{1+\theta}_{\sharp}([0,T];\mathbf{H}^{-1/2}(\Gamma_{i}))}+\norme{\omega_{2}}_{\mathcal{C}^{\theta}_{\sharp}([0,T];H^{1/2}(\Gamma_{o}))}\leq R,
\]
the system \eqref{chap2-existence1} admits a $T$-periodic strict solution $(\ug,p,\eta)$ belonging to (after a change of variables mapping $\Omega_{\eta(t)}$ into $\Omega$)
\begin{itemize}
\item $\ug\in \mathcal{C}^{\theta}_{\sharp}([0,T];\textbf{H}^{2}(\Omega))\cap \mathcal{C}^{1+\theta}_{\sharp}([0,T];\BS{L}^{2}(\Omega))$.
\item $p\in \mathcal{C}^{\theta}_{\sharp}([0,T];H^{1}(\Omega))$.
\item $\eta\in \mathcal{C}^{\theta}_{\sharp}([0,T];H^{4}(\Gamma_{s})\cap H^{2}_{0}(\Gamma_{s}))\cap \mathcal{C}^{1+\theta}_{\sharp}([0,T];H^{2}_{0}(\Gamma_{s}))\cap\mathcal{C}^{2+\theta}_{\sharp}([0,T];L^{2}(\Gamma_{s}))$.
\end{itemize}
\end{theorem}
The functional spaces are introduced in Section \ref{function.spaces}. In the appendix we present existence results for time-periodic abstract evolution equation. For a periodic evolution equation
\[
\begin{cases}
\begin{aligned}
&y'(t)=Ay(t)+f(t),\text{ for }t\in[0,T],\\
&y(0)=y(T),
\end{aligned}
\end{cases}
\]
with $T>0$, the existence of a solution is related to the spectral criteria $1\in \rho(S(T))$ where $(S(t))_{t\geq 0}$ is the semigroup associated with $A$. This simple criteria follows from the Duhamel formula and is well known. It is stated, for example, in \cite{MR1204883,MR1040464} for $T$-periodic mild solutions and in \cite{MR963847,MR929914} for strict solutions in $\mathcal{C}^{0}$ with H\"older regularity in time (and for time-dependent operator $A(t)$). Our approach, however, specifies the different regularities that we can expect on the periodic solution, depending on the source term $f$. We also provide explicit conditions on the pair $(A,T)$ to ensure that the spectral criteria is satisfied. Remark that the previous results always assume that $A$ is the infinitesimal generator of an analytic semigroup. For abstract periodic evolution equations with weaker assumptions on $A$ we refer to \cite{MR1937154}.

Let us conclude this introduction with a brief history on the existence of time-periodic solutions for the Navier--Stokes equations. This question was initially considered in 1960s in \cite{MR0120968,MR0115017,MR0170541,MR0105251}. In particular, in \cite{MR0120968,MR0115017,MR0170541}, the authors obtained a periodic weak solution by considering a fixed point of the Poincar\'e map which takes an initial value and provides the state of the corresponding initial-value problem at time $T$. The existence of strong solutions for small data is proved in \cite{MR0213714} in 3D and without size rectriction in \cite{MR0264254} in 2D. For more recent results with non-homogeneous boundary conditions see \cite{MR2589952,MR2861830}. The existence theory for the periodic Navier--Stokes equations in bounded domain is now as developed as the existence theory for the initial value problem. For unbounded domain the question is still delicate and was investigated, with zero boundary conditions at the infinity, in \cite{MR2062429,MR3039695,MR1373173,MR1107016,MR1699121,MR1777114}. For further references on the existence of periodic solutions for the Navier--Stokes equations we refer to \cite{tuprints3309}.

The method developed in this article corresponds to the Poincar\'e map approach, applied on the whole coupled fluid--structure system. Note that the periodic solution obtained for the Navier--Stokes equations is usually unique. Here the free boundary makes the analysis of the uniqueness more complicated. For instance, we cannot considered the difference of two periodic solutions in their respective time-dependent domains, which may be different. The difference has to be taken after a change of variables mapping both periodic solutions in the same domain. In that case, energy estimates are difficult to obtain due to the higher order `geometrical' nonlinear terms. The uniqueness question remains an open question in our work.
\subsection{Equivalent system in a reference configuration}
To fix the domain we perform the following change of variables
\begin{equation}\label{chap2-changevarrectangle}
\mathcal{T}_{\eta,0}(t) :
\begin{cases}
\Omega_{\eta(t)}&\longrightarrow \Omega,\\
(x,y)&\longmapsto (x,z)=\left(x,\frac{y}{1+\eta(x,t)}\right).\\
\end{cases}
\end{equation}
Setting $Q_{T}=\Omega\times(0,T)$, $\hat{\BS{u}}(x,z,t)=\ug(\mathcal{T}_{\eta,0}^{-1}(t)(x,z),t)$ and $\hat{p}(x,z,t)=p(\mathcal{T}_{\eta,0}^{-1}(t)(x,z),t)$, the system (\ref{chap2-existence1}) becomes
\begin{equation}\label{chap2-mainexistence}
\begin{aligned}
&\hat{\boldsymbol{u}}_{t}-\nu\Delta\hat{\BS{u}}+\nabla\hat{p}=\BS{g}(\hat{\BS{u}},\hat{p},\eta),\,\,\,\DIV\hat{\BS{u}}=\DIV\BS{w}(\hat{\BS{u}},\eta)\,\text{ in }Q_{T},\\
&\hat{\BS{u}}=\eta_{t}\textbf{e}_{2}\,\text{ on }\Sigma^{s}_{T},\\
&\hat{\BS{u}}=\boldsymbol{\omega}_{1}\,\text{ on }\Sigma^{i}_{T},\\
&\hat{u}_2=0\,\text{ and }\,\hat{p}+(1/2)\vert\hat{\BS{u}}\vert^{2}=\omega_{2}\,\text{ on }\Sigma^{o}_{T},\\
&\hat{\BS{u}}=0\,\text{ on }\Sigma^{b}_{T},\,\,\,\hat{\BS{u}}(0)=\hat{\BS{u}}(T)\,\text{ on }\Omega\\
&\eta_{tt}-\beta\eta_{xx}-\gamma\eta_{txx}+\alpha\eta_{xxxx}=\hat{p}-2\nu\hat{u}_{2,z}+\Psi(\hat{\BS{u}},\eta)\,\text{ on }\Sigma^{s}_{T},\\
&\eta=0\,\text{ and }\,\eta_{x}=0\,\text{ on }\{0,L\}\times(0,T),\\
&\eta(0)=\eta(T)\,\text{ and }\,\eta_{t}(0)=\eta_{t}(T)\,\text{ in }\Gamma_{s},
\end{aligned}
\end{equation}
with
\begin{align*}
&\begin{aligned}
\BS{G}(\hat{\BS{u}},\hat{p},\eta)={}&-\eta\hat{\BS{u}}_t+\left[z\eta_t+\nu z\left(\frac{\eta_x^2}{1+\eta}-\eta_{xx}\right)\right]\hat{\BS{u}}_z\\
&+\nu\left[-2z\eta_x\hat{\BS{u}}_{xz}+\eta\hat{\BS{u}}_{xx}+\frac{z^2\eta_x^2-\eta}{1+\eta}\hat{\BS{u}}_{zz}\right]+z\eta_x\hat{p}_z\textbf{e}_{1}\\
&-z\eta\hat{p}_x\textbf{e}_1-(1+\eta)\hat{u}_1\hat{\BS{u}}_x+(z\eta_x
\hat{u}_1-\hat{u}_2)\hat{\BS{u}}_z,
\end{aligned}\\
&\begin{aligned}
\BS{w}[\hat{\BS{u}},\eta]={}&-\eta\hat{u}_1\textbf{e}_1+z\eta_x\hat{u}_1\textbf{e}_2,\\
\end{aligned}\\
&\begin{aligned}
&\Psi(\hat{\BS{u}},\eta)=\nu\left(\frac{\eta_x}{1+\eta}\hat{u}_{1,z}+\eta_x\hat{u}_{2,x}-\frac{\eta_{x}^{2}z-2\eta}{1+\eta}\hat{u}_{2,z}\right).
\end{aligned}\\
\end{align*}
We study the linear periodic system associated to \eqref{chap2-mainexistence} in Section \ref{section-semigroup-formulation}--\ref{chap2-time-periodic-linear}. The existence of time-periodic solution for \eqref{chap2-mainexistence} is established in Section \ref{chap2-nonlinear} with a fixed point procedure.
\subsection{Function spaces}\label{function.spaces}
To deal with the mixed boundary conditions introduce the spaces
\[\boldsymbol{V}^{0}_{n,\Gamma_{d}}(\Omega)=\{\vg\in \BS{L}^{2}(\Omega)\mid \DIV\vg=0\text{ in }\Omega, \vg\cdot\textbf{n}=0\text{ on }\Gamma_{d}\},\]
and the orthogonal decomposition of $\BS{L}^{2}(\Omega)=L^{2}(\Omega,\mathbb{R}^{2})$
\[\BS{L}^{2}(\Omega)=\BS{V}^{0}_{n,\Gamma_{d}}(\Omega)\oplus\text{grad }H^{1}_{\Gamma_{o}}(\Omega),\]
where $H^{1}_{\Gamma_{o}}(\Omega)=\{u\in H^{1}(\Omega)\mid u=0 \text{ on }\Gamma_{o}\}$.  Let $\Pi :\BS{L}^{2}(\Omega)\rightarrow \BS{V}^{0}_{n,\Gamma_{d}}(\Omega)$ be the so-called Leray projector associated with this decomposition. If $\ug$ belongs to $\BS{L}^{2}(\Omega)$ then $\Pi\textbf{u}=\ug-\nabla p_{\ug}-\nabla q_{\ug}$ where $p_{\ug}$ and $q_{\ug}$ are solutions to the following elliptic equations
\begin{equation}\label{chap2-equation-Pif}
\begin{aligned}
&p_{\ug}\in H^{1}_{0}(\Omega),\,\,\,\Delta p_{\ug}=\DIV\ug\in H^{-1}(\Omega),\\
&q_{\ug}\in H^{1}_{\Gamma_{o}}(\Omega),\,\,\,\Delta q_{\ug}=0,\,\,\,\frac{\partial q_{\ug}}{\partial \normal}=(\ug-\nabla p_{\ug})\cdot\normal\,\text{ on }\Gamma_{d},\,\,\,q_{\ug}=0\,\text{ on }\Gamma_{o}.
\end{aligned}
\end{equation}
Throughout this article the functions and spaces with vector values are written with a bold typography. For example $\BS{H}^{2}(\Omega)=H^{2}(\Omega,\Real^{2})$. Using the notations in \cite[Theorem 11.7]{MR0350177}, we introduce the space $H^{3/2}_{00}(\Gamma_{s})=[H^{1}_{0}(\Gamma_{s}),H^{2}_{0}(\Gamma_{s})]_{1/2}$. This space is a strict subspace of $H^{3/2}_{0}(\Gamma_{s})=H^{3/2}(\Gamma_{s})\cap H^{1}_{0}(\Gamma_{s})$. Odd and even symmetries preserve the $H^{k}$-regularity for functions in $H^{k}_{0}(\Gamma_{s})$ with $k=1,2$, thus, by interpolation, the $H^{3/2}$-regularity is also preserved for functions in $H^{3/2}_{00}(\Gamma_{s})$. This property is used in \cite{2017arXiv170706382C} to handle the pressure boundary condition.

For the boundary condition on the inflow, we use the results developed in \cite{MR2641539} for elliptic equations in a dihedron. In our case, the angle between $\Gamma_{i}$ and $\Gamma_{s}$ is equal to $\frac{\pi}{2}$. If $\boldsymbol{\omega}$ (resp. $\BS{g}$) denotes the boundary condition on $\Gamma_{i}$ (resp. $\Gamma_{s,0}$), the Laplace and Stokes equations possess solutions with $H^{2}$-regularity near $C_{0,1}=(0,1)$ provided that the data are regular enough and that the compatibility conditions $\boldsymbol{\omega}(C_{0,1})=\BS{g}(C_{0,1})$ is satisfied. To ensure these conditions, the non-homogeneous boundary condition on $\Gamma_{i}$ is chosen in $H^{3/2}_{0}(\Gamma_{i})$. Consider the Stokes system
\begin{equation}\label{chap2-Stokes-def}
\begin{aligned}
&-\nu\Delta \ug +\nabla p=\BS{f},\,\,\,\,\DIV\ug=0\,\text{ in }\Omega,\\
&\ug=0\,\text{ on }\Gamma_{d},\,\,\,u_{2}=0\,\text{ and }p=0\,\text{ on }\Gamma_{o}.\\
\end{aligned}
\end{equation}
The energy space associated with (\ref{chap2-Stokes-def}) is
\[V=\{\ug\in \BS{H}^{1}(\Omega)\mid \DIV\ug=0\text{ in }\Omega\text{, }\ug=0\text{ on }\Gamma_{d}\text{, }u_2=0\text{ on }\Gamma_{o}\}.\]
The regularity result for \eqref{chap2-Stokes-def} is similar to \cite[Theorem 5.4]{2017arXiv170706382C} and we define the Stokes operator $(A_{s},\mathcal{D}(A_{s}))$ in $\BS{V}^{0}_{n,\Gamma_{d}}(\Omega)$ by
\[\mathcal{D}(A_s)=\BS{H}^{2}(\Omega)\cap V,\,\text{ and for all}\,\ug\in \mathcal{D}(A_s),\,A_s\ug=\nu\Pi\Delta\ug.\]
We also introduce the space $\BS{V}^{s}(\Omega)=\{\ug\in \BS{H}^{s}(\Omega)\mid \DIV\ug=0\}$ for $s\geq 0$. To describe the Dirichlet boundary condition on $\Gamma_{s}$ set
\begin{align*}
\mathcal{L}^{2}(\Gamma_{s})&=\{0\}\times L^{2}(\Gamma_{s}),&\mathcal{H}^{3/2}_{00}(\Gamma_{s})&=\{0\}\times H^{3/2}_{00}(\Gamma_{s}),\\
\mathcal{H}^{\kappa}(\Gamma_{s})&=\{0\}\times H^{\kappa}(\Gamma_{s}),&\mathcal{H}^{\kappa}_{0}(\Gamma_{s})&=\{0\}\times H^{\kappa}_{0}(\Gamma_{s})
\,\text{ for }\kappa\geq 0.
\end{align*}
For $\kappa\geq 0$, the dual space of $\mathcal{H}^{\kappa}(\Gamma_{s})$ with $\mathcal{L}^{2}(\Gamma_{s})$ as pivot space is denoted by $(\mathcal{H}^{\kappa}(\Gamma_{s}))'$.

For space-time dependent functions we use the notations introduced in \cite{MR0350178}:
\begin{align*}
\BS{L}^{2}(Q_{T})&{}=L^{2}(0,T;\BS{L}^{2}(\Omega)),\,\,\,\BS{H}^{p,q}(Q_{T})=L^{2}(0,T;\BS{H}^{p}(\Omega))\cap H^{q}(0,T;\BS{L}^{2}(\Omega)),\,p,q\geq 0,\\
L^{2}(\Sigma_{T}^{s})&{}=L^{2}(0,T;L^{2}(\Gamma_{s})),\,\,\,\,H^{p,q}(\Sigma^{s}_{T})=L^{2}(0,T;H^{p}(\Gamma_{s}))\cap H^{q}(0,T;L^{2}(\Gamma_{s})),\,\,\,p,q\geq 0.
\end{align*}
If $X$ is a space of functions and $\rho\geq 0$ we set
\begin{align*}
&\mathcal{C}^{\rho}_{\sharp}([0,T];X):=\{v_{\vert [0,T]}\mid v\in \mathcal{C}^{\rho}(\mathbb{R};X)\text{ is }T\text{-periodic}\},\\
&H^{\rho}_{\sharp}(0,T;X):=\{v_{\vert [0,T]}\mid v\in H^{\rho}_{\text{loc}}(\mathbb{R};X)\text{ is }T\text{-periodic}\}.
\end{align*}

\section{Linear system}\label{section2}
\subsection{Stokes system with non-homogeneous mixed boundary conditions}
In this section we consider the Stokes system
\begin{equation}\label{chap2-stokes-nh}
\begin{aligned}
&\lambda\ug-\nu\Delta\ug+\nabla p=\BS{f},\,\,\,\DIV\ug=0\text{ in }\Omega,\\
&\ug=\BS{g}\,\text{ on }\Gamma_{s},\,\,\, \ug=\boldsymbol{\omega}\,\text{ on }\Gamma_{i},\\
&u_{2}=0\,\text{ and }p=0\text{ on }\Gamma_{o},\,\,\,\ug=0\,\text{ on }\Gamma_{b},
\end{aligned}
\end{equation}
with $\lambda\in\mathbb{C}$, $\BS{f}\in\BS{L}^{2}(\Omega)$, $\BS{g}\in\mathcal{H}^{3/2}_{00}(\Gamma_{s})$ and $\boldsymbol{\omega}\in \BS{H}^{3/2}_{0}(\Gamma_{i})$. The following lemmas provide suitable lifting of the non-homogeneous Dirichlet boundary conditions on $\Gamma_{s}$ and $\Gamma_{i}$.
\begin{lemma}\label{chap2-lemme-gammas}
There exists $\Phi_{s}\in\mathcal{L}(\mathcal{H}^{3/2}_{00}(\Gamma_{s}),\BS{H}^{2}(\Omega))$ such that, for all $\BS{g}\in \mathcal{H}^{3/2}_{00}(\Gamma_{s})$, $\wg=\Phi_s(\BS{g})$ satisfies
\begin{equation}\label{chap2-liftS}
\begin{aligned}
&\DIV\wg=0\text{ in }\Omega,\\
&\wg=\BS{g}\text{ on }\Gamma_{s},\,\wg=0\text{ on }\Gamma_{i}\cup\Gamma_{b},\,w_{2}=0\text{ on }\Gamma_{o}.
\end{aligned}
\end{equation}
\end{lemma}
\begin{proof}
The idea to solve \eqref{chap2-liftS} is to use a Stokes system with Dirichlet boundary conditions on an extended domain. We set $\Omega_{e}=(0,2L)\times(0,1)$, $\Gamma_{s,e}=(0,2L)\times\{1\}$, $\Gamma_{b,e}=(0,2L)\times\{0\}$, $\Gamma_{o,e}=\{2L\}\times(0,1)$ and
\[\hat{\BS{g}}:\begin{cases}
\begin{array}{l}
\hat{\BS{g}}=\BS{g}\text{ on }(0,L)\times\{1\},\\
\hat{\BS{g}}(x,1)=-\BS{g}(2L-x,1)\text{ for }x\in(L,2L).
\end{array}
\end{cases}
\]  
Thanks to the properties of the space $H^{3/2}_{00}(\Gamma_{s})$ with respect to symmetries, the function $\hat{\BS{g}}$ is in $\mathcal{H}^{3/2}_{00}(\Gamma_{s,e})$. Moreover, it has a zero average by construction. Consider the Stokes system
\begin{equation}\label{chap3-sys.expl}
\begin{aligned}
&-\nu\Delta\vg + \nabla q=0,\,\,\,\DIV\vg=0\,\text{ in }\Omega_{e},\\
&\vg=\hat{\BS{g}}\,\text{ on }\Gamma_{s,e},\,\vg=0\,\text{ on }\partial\Omega_{e}\setminus\Gamma_{s,e}.
\end{aligned}
\end{equation}
This system admits a unique solution $(\vg,q)\in \BS{H}^{2}(\Omega_{e})\times H^{1}(\Omega_{e})$ (see for example \cite{MR2641539}; note that one could not find $\wg$ directly by solving \eqref{chap3-sys.expl} on $\Omega$, since $\BS{g}$ does not necessarily have a zero average on $\Gamma_s$, contrary to $\hat{\BS{g}}$ on $\Gamma_{s,e}$. We introduce the function
\[\vg_{s}(x,y):=\begin{pmatrix}
1&0\\
0&-1
\end{pmatrix}\vg(2L-x,y)\text{ for all }(x,y)\in\Omega_{e}.
\]
The function $\vg_{s}\in\BS{H}^{2}(\Omega_{e})$ still satisfies 
\[\begin{aligned}
&\DIV\vg_{s}=0\text{ in }\Omega_{e},\\
&\vg_{s}=\hat{\BS{g}}\text{ on }\Gamma_{s,e},\,\vg_{s}=0\text{ on }\partial\Omega_{e}\setminus\Gamma_{s,e},
\end{aligned}
\]
and $\hat{\vg}:=\frac{\vg+\vg_{s}}{2}$ verifies $\hat{v}_{2}(L,y)=0$ for all $y\in(0,1)$. The restriction to $\Omega$ of $\hat{\vg}$ is solution to \eqref{chap2-liftS}. The linearity of the mapping $\BS{g}\mapsto \wg$ is obvious from the construction above, and its continuity (that is, an estimate $\norme{\wg}_{\BS{H}^{2}(\Omega)}\leq C\norme{\BS{g}}_{\mathcal{H}^{3/2}_{00}(\Gamma_{s})}$) follows from the classical estimates for the Stokes system with Dirichlet boundary conditions.
\end{proof}
\begin{lemma}\label{chap2-lemme-gammai}
There exists $\Phi_{i}\in\mathcal{L}(\BS{H}^{3/2}_{0}(\Gamma_{i}),\BS{H}^{2}(\Omega))$ such that, for all $\boldsymbol{\omega}\in \BS{H}^{3/2}_{0}(\Gamma_{i})$, $\wg=\Phi_i(\boldsymbol{\omega})$ satisfies
\begin{equation}\label{chap2-liftI}
\begin{aligned}
&\DIV\wg=0\text{ in }\Omega,\\
&\wg=\boldsymbol{\omega}\text{ on }\Gamma_{i},\,\wg=0\text{ on }\Gamma_{s}\cup\Gamma_{b},\,w_{2}=0\text{ on }\Gamma_{o}.
\end{aligned}
\end{equation}
\end{lemma}
\begin{proof}
Once again we construct $\wg$ by solving a Stokes system with Dirichlet boundary conditions. First, we have to compensate the non-zero average of $\boldsymbol{\omega}\cdot\normal$ on $\Gamma_{i}$. Consider the function $\boldsymbol{\omega}^{-}\in \mathcal{H}^{3/2}_{00}(\Gamma_{s})$ defined by
\[\boldsymbol{\omega}^{-}(x)=-\frac{\varphi(x)}{\int_{\Gamma_{s}}\varphi}\left(\int_{\Gamma_{i}}\boldsymbol{\omega}\cdot\normal\right)\textbf{e}_{2},\quad\forall x\in(0,L),\]
where $\varphi\in \mathcal{C}^{\infty}_{0}(\Gamma_{s})$ satisfies $\int_{\Gamma_{s}}\varphi\neq 0$.
Consider then the system
\[\begin{aligned}
&-\nu\Delta\vg + \nabla q =0,\,\,\,\DIV\vg =0 \text{ in }\Omega,\\
&\vg =\boldsymbol{\omega}\text{ on }\Gamma_{i},\,\vg=\boldsymbol{\omega}^{-}\text{ on }\Gamma_{s},\,\vg=0\text{ on }\Gamma_{b}\cup\Gamma_{o}.
\end{aligned}
\]
Using \cite{MR2641539}, we obtain a solution $(\vg,q)\in \BS{H}^{2}(\Omega)\times H^{1}(\Omega)$ to this system. Finally $\wg=\vg-\Phi_{s}(\boldsymbol{\omega}^{-})$ satisfies \eqref{chap2-liftI}. Once again, the linearity of $\Phi_i:\boldsymbol{\omega}\mapsto \wg$ is trivial by construction, and its continuity follows from the classical estimates for the Stokes equations with Dirichlet boundary conditions, and from the construction of $\boldsymbol{\omega}^{-}$.
\end{proof}

We can now specify the regularity results for \eqref{chap2-stokes-nh}.
\begin{theorem}\label{chap2-regularity.stokes}
For all $(\BS{f},\BS{g},\boldsymbol{\omega})\in \BS{L}^{2}(\Omega)\times \mathcal{H}^{3/2}_{00}(\Gamma_{s})\times \BS{H}^{3/2}_{0}(\Gamma_{i})$, \eqref{chap2-stokes-nh} admits a unique solution $(\ug,p)\in\BS{H}^{2}(\Omega)\times H^{1}(\Omega)$ which satisfies
\[\norme{\ug}_{\BS{H}^{2}(\Omega)}+\norme{p}_{H^{1}(\Omega)}\leq C(\norme{\BS{f}}_{\BS{L}^{2}(\Omega)}+\norme{\BS{g}}_{\mathcal{H}^{3/2}_{00}(\Gamma_{s})}+\norme{\boldsymbol{\omega}}_{\BS{H}^{3/2}_{0}(\Gamma_{i})}).\]
\end{theorem}
\begin{proof}
Consider $\vg=\ug -\Phi_{s}(\BS{g})-\Phi_{i}(\boldsymbol{\omega})$. The pair $(\vg,p)$ is solution to
\[
\begin{aligned}
&\lambda\vg-\nu\Delta\vg+\nabla p=\hat{\BS{f}},\,\,\,\DIV\vg=0\text{ in }\Omega,\\
&\vg=0\,\text{ on }\Gamma_{d},\,u_{2}=0\,\text{ and }p=0\text{ on }\Gamma_{o},
\end{aligned}
\]
with $\hat{\BS{f}}=\BS{f}+\nu\Delta\Phi_{s}(\BS{g})+\nu\Delta\Phi_{i}(\boldsymbol{\omega})-\lambda\Phi_{s}(\BS{g})-\lambda\Phi_{i}(\boldsymbol{\omega}) \in \BS{L}^2(\Omega)$. The $\BS{H}^{2}$-regularity of $\vg$ in a neighbourhood of $\Gamma_{i}$ is well known for Stokes with homogeneous Dirichlet conditions. The lower order term $\lambda\vg$ does not impact the regularity of the system and can be dealt with a bootstrap argument. The regularity on a neighbourhood of $\Gamma_{o}$ is proved in \cite[Theorem 5.4]{2017arXiv170706382C}. Hence, $(\vg,p)\in\BS{H}^{2}(\Omega)\times H^{1}(\Omega)$, and thus $(\ug,p)\in\BS{H}^{2}(\Omega)\times H^{1}(\Omega)$ with the desired estimates.
\end{proof}
We introduce the lifting operators:
\begin{itemize}
\item $L\in \mathcal{L}(\mathcal{H}^{3/2}_{00}(\Gamma_{s}),\BS{H}^{2}(\Omega)\times H^{1}(\Omega))$ defined by
\[L(\BS{g})=(L_{1}(\BS{g}),L_{2}(\BS{g}))=(\BS{w}_{1},\rho_{1}),\]
where $(\BS{w}_{1},\rho_{1})$ is solution to (\ref{chap2-stokes-nh}) with $(\BS{f},\BS{g},\boldsymbol{\omega})=(\boldsymbol{0},\BS{g},\boldsymbol{0})$ and $\lambda=0$.
\item $L_{\Gamma_{i}}\in \mathcal{L}(\BS{H}^{3/2}_{0}(\Gamma_{i}),\BS{H}^{2}(\Omega)\times H^{1}(\Omega))$ defined by
\[L_{\Gamma_{i}}(\boldsymbol{\omega})=(L_{\Gamma_{i},1}(\boldsymbol{\omega}),L_{\Gamma_{i},2}(\boldsymbol{\omega}))=(\BS{w}_{2},\rho_2),\]
where $(\BS{w}_{2},\rho_2)$ is the solution to (\ref{chap2-stokes-nh}) with $(\BS{f},\BS{g},\boldsymbol{\omega})=(\boldsymbol{0},\boldsymbol{0},\boldsymbol{\omega})$ and $\lambda=0$.
\item $L_{\Gamma_{o}}\in\mathcal{L}(H^{1/2}(\Gamma_{o}),H^{1}(\Omega))$ a continuous lifting operator.
\end{itemize}
In order to express the pressure, we also consider the operators:
\begin{itemize}
\item $N_{s}\in\mathcal{L}(\mathcal{H}^{3/2}_{00}(\Gamma_{s}),H^{3}(\Omega))$ defined by $N_{s}(\BS{g})=p_{1}$ with 
\begin{equation}
\begin{aligned}
&\Delta p_{1}=0\,\text{ in }\Omega,\\
&\frac{\partial p_{1}}{\partial\normal}=\BS{g}\cdot\normal\,\text{ on }\Gamma_{s},\\
&\frac{\partial p_{1}}{\partial\normal}=0\,\text{ on }\Gamma_{i}\cup\Gamma_{b},\\
&p_{1}=0\,\text{ on }\Gamma_{o}.
\end{aligned}
\end{equation}
\item $N_{v}\in\mathcal{L}(\BS{H}^{2}(\Omega),H^{1}(\Omega))$ defined by $N_{v}(\ug)=p_{2}$ with
\begin{align*}
&\Delta p_{2}=0\,\text{ in }\Omega,\\
&\frac{\partial p_{2}}{\partial\normal}=\nu\Delta\Pi\ug\cdot\normal\,\text{ on }\Gamma_{d},\\
&p_{2}=0\,\text{ on }\Gamma_{o}.
\end{align*}
\item $N_{p}\in\mathcal{L}(\BS{L}^{2}(\Omega),H^{1}_{\Gamma_{o}}(\Omega))$ defined by $N_{p}(\BS{f})=p_{3}$ with $(I-\Pi)\BS{f}=\nabla p_{3}$.
\end{itemize}
\begin{lemma}\label{chap2-reg-Ns-Ndiv}
The operator $N_{s}$ can be extended as follows:
\begin{itemize}
\item $N_{s}\in \mathcal{L}((\mathcal{H}^{3/2}(\Gamma_{s}))',L^{2}(\Omega))$.
\item $N_{s}\in\mathcal{L}((\mathcal{H}^{1/2}(\Gamma_{s}))',H^{1}(\Omega))$.
\end{itemize}
\end{lemma}
\begin{proof}
The first result is obtained by duality. The second follows from interpolation techniques.
\end{proof}

To prepare the matrix formulation of the fluid--structure system, we recast the Stokes system in terms of $\Pi\ug$ and $(I-\Pi)\ug$.

\begin{theorem}\label{chap2-reformulation.stokes}
Suppose that $\boldsymbol{\omega}=0$ and $(\BS{f},\BS{g})\in\BS{L}^{2}(\Omega)\times\mathcal{H}^{3/2}_{00}(\Gamma_{s})$. A pair $(\ug,p)$ is solution to \eqref{chap2-stokes-nh} if and only if
\begin{equation}\label{chap2-equiv.stokes}
\begin{aligned}
&\lambda\Pi\ug-A_{s}\Pi\ug + A_{s}\Pi L_{1}(\BS{g})=\Pi\BS{f},\\
&(I-\Pi)\ug=\nabla N_{s}(\BS{g}),\\
&p=-\lambda N_{s}(\BS{g})+N_{v}(\Pi\ug)+N_{p}(\BS{f}).
\end{aligned}
\end{equation}
\end{theorem}
\begin{proof}
Remark that $\ug - L_{1}(\BS{g})$ belongs to $\mathcal{D}(A_{s})$ and
\begin{multline}\label{chap2-touse}
-\nu\Pi\Delta\ug=-\nu\Pi\Delta(\ug-L_{1}(\BS{g}))+\nu\Pi\Delta L_{1}(\BS{g})=-A_{s}\Pi(\ug-L_{1}(\BS{g}))=-A_{s}\Pi\ug + A_{s}\Pi L_{1}(\BS{g}).
\end{multline}
In the previous identities we have used the extrapolation method to extend $A_{s}$ as an unbounded operator in $\mathcal{D}(A_{s}^{*})'$ with domain $\BS{V}^{0}_{n,\Gamma_{d}}(\Omega)$. Applying $\Pi$ on the first line of \eqref{chap2-stokes-nh} we obtain
\[\lambda\Pi\ug -\nu\Pi\Delta\ug=\Pi\BS{f},\]
which, using \eqref{chap2-touse}, provides the first line in \eqref{chap2-equiv.stokes}. The second line follows directly from the elliptic equations \eqref{chap2-equation-Pif} used to compute $(I-\Pi)\ug$. Finally the pressure is obtained by applying $(I-\Pi)$ to the first line of \eqref{chap2-stokes-nh}.
\end{proof}
\subsection{Beam equation}
Let $(A_{\alpha,\beta},\mathcal{D}(A_{\alpha,\beta}))$ be the unbounded operator in $L^{2}(\Gamma_s)$ defined by $\mathcal{D}(A_{\alpha,\beta})=H^{4}(\Gamma_s)\cap H^{2}_{0}(\Gamma_s)$ and, for all $\eta\in\mathcal{D}(A_{\alpha,\beta})$, $A_{\alpha,\beta}\eta=\beta\eta_{xx}-\alpha\eta_{xxxx}$. The operator $A_{\alpha,\beta}$ is self-adjoint and is an isomorphism from $\mathcal{D}(A_{\alpha,\beta})$ to $L^{2}(\Gamma_s)$.

The space $H^{2}_{0}(\Gamma_{s})$ is equipped with the inner product
\[\langle \eta_{1},k_{1}\rangle_{H^{2}_{0}(\Gamma_{s})}=\int_{\Gamma_{s}}(-A_{\alpha,\beta})^{1/2}\eta_{1}(-A_{\alpha,\beta})^{1/2}k_{1}.\]
The unbounded operator $(A_{b},\mathcal{D}(\mathcal{A}_{b}))$ associated with the beam, in
$H_{b}=H^{2}_{0}(\Gamma_s)\times L^{2}(\Gamma_s)$,
is defined by 
\[
\mathcal{D}(\mathcal{A}_{b})=(H^{4}(\Gamma_s)\cap H^{2}_{0}(\Gamma_s))\times H^{2}_{0}(\Gamma_s)
\mbox{ and }\mathcal{A}_{b}=
\begin{pmatrix}
0&I\\
A_{\alpha,\beta}&\gamma\Delta_s
\end{pmatrix}.
\]
\begin{theorem}The operator $(\mathcal{A}_{b},\mathcal{D}(\mathcal{A}_{b}))$ is the infinitesimal generator of an analytic semigroup on $H_{b}$.
\end{theorem}
\begin{proof}
See \cite[Theorem 1.1]{MR971932}.
\end{proof}
\subsection{Semigroup formulation of the linear fluid--structure system}\label{section-semigroup-formulation}
Consider a period $T>0$. Set $\theta\in(0,1)$ and 
\[(\boldsymbol{\omega_{1}},\omega_{2})\in \left(\mathcal{C}^{\theta}_{\sharp}([0,T];\BS{H}^{3/2}_{0}(\Gamma_{i}))\cap \mathcal{C}^{1+\theta}_{\sharp}([0,T];\BS{H}^{-1/2}(\Gamma_{i}))\right)\times\mathcal{C}^{\theta}_{\sharp}([0,T];H^{1/2}(\Gamma_{o})).\]
For $(\BS{f},\Theta,h)$ in $\mathcal{C}^{\theta}_{\sharp}(0,T;\BS{L}^{2}(\Omega))\times \mathcal{C}^{\theta}_{\sharp}([0,T];H^{1/2}(\Gamma_{o}))\times \mathcal{C}^{\theta}_{\sharp}([0,T];L^{2}(\Gamma_{s}))$ and \[\BS{w}\in\mathcal{C}^{1+\theta}_{\sharp}([0,T];\BS{L}^{2}(\Omega))\cap \mathcal{C}^{\theta}_{\sharp}([0,T];\BS{H}^{2}(\Omega)\cap \BS{H}^{1}_{0}(\Omega)),\]
consider the following linear system
\begin{equation}\label{chap2-linear-existence1}
\begin{aligned}
&\ug_{t}-\nu\Delta\ug+\nabla p=\BS{f},\,\,\,\DIV\ug=\DIV\BS{w}\,\text{ in }Q_{T},\\
&\ug=\eta_{t}\textbf{e}_{2}\,\text{ on }\Sigma^{s}_{T},\,\,\,
\ug=\boldsymbol{\omega}_{1}\,\text{ on }\Sigma^{i}_{T},\\
&u_2=0\text{ and }p=\omega_{2}+\Theta\,\text{ on }\Sigma^{o}_{T},\\
&\ug=0\,\text{ on }\Sigma^{b}_{T},\,\,\,\ug(0)=\ug(T)\,\text{ in }\Omega,\\
&\eta_{tt}-\beta\eta_{xx}-\gamma\eta_{txx}+\alpha\eta_{xxxx}=p-2\nu u_{2,z}+h\,\text{ in }\Sigma^{s}_{T},\\
&\eta=0\,\text{ and }\,\eta_{x}=0\,\text{ on }\{0,L\}\times(0,T),\\
&\eta(0)=\eta(T)\,\text{ and }\,\eta_{t}(0)=\eta_{t}(T)\,\text{ in }\Gamma_{s}.
\end{aligned}
\end{equation}
For a scalar function $\eta$ defined on $\Gamma_{s}$ we use the notation $L_{1}(\eta)=L_{1}(\eta\textbf{e}_{2})$. We look for a solution to (\ref{chap2-linear-existence1}) under the form $(\ug,p,\eta)=(\vg,q,\eta) + (\BS{w}+L_{\Gamma_{i},1}(\boldsymbol{\omega}_{1}),L_{\Gamma_{o}}(\omega_{2})+L_{\Gamma_{o}}(\Theta)+L_{\Gamma_{i},2}(\boldsymbol{\omega}_{1}),0)$ with $(\vg,q,\eta)$ solution to
\begin{equation}\label{chap2-linear-existence2}
\begin{aligned}
&\BS{v}_{t}-\nu\Delta\vg + \nabla q = F,\,\,\,\DIV\BS{v}=0\,\text{ in }Q_{T},\\
&\BS{v}=\eta_{t}\textbf{e}_{2}\,\text{ on }\Sigma^{s}_{T},\,\,\,
\vg=0\,\text{ on }\Sigma^{i}_{T},\\
&v_2=0\text{ and }q=0\,\text{ on }\Sigma^{o}_{T},\\
&\vg=0\,\text{ on }\Sigma^{b}_{T},\,\,\,\BS{v}(0)=\vg(T)\,\text{ in }\Omega,\\
&\eta_{tt}-\beta\eta_{xx}-\gamma\eta_{txx}+\alpha\eta_{xxxx}=q+H\,\text{ in }\Sigma^{s}_{T},\\
&\eta=0\,\text{ and }\,\eta_{x}=0\,\text{ on }\{0,L\}\times(0,T),\\
&\eta(0)=\eta(T)\,\text{ and }\,\eta_{t}(0)=\eta_{t}(T)\,\text{ in }\Gamma_{s},
\end{aligned}
\end{equation}
where $F=\BS{f}-\BS{w}_{t}+\nu\Delta \BS{w}-\partial_{t}L_{\Gamma_{i},1}(\boldsymbol{\omega}_{1})-\nabla L_{\Gamma_{o}}(\omega_{2})-\nabla L_{\Gamma_{o}}(\Theta)$ and $H=w_{2,z}+L_{\Gamma_{i},2}(\boldsymbol{\omega}_{1})+L_{\Gamma_{o}}(\omega_{2})+L_{\Gamma_{o}}(\Theta)+h$. 
\begin{theorem}\label{chap2-lin.recast}
Suppose that $\eta_{t}\in \mathcal{C}^{1+\theta}_{\sharp}([0,T];L^{2}(\Gamma_{s}))\cap \mathcal{C}^{\theta}_{\sharp}([0,T];H^{2}_{0}(\Gamma_{s}))$. A pair
\begin{equation}\label{chap2-reg.linper}
(\vg,q)\in \left(\mathcal{C}^{1+\theta}_{\sharp}([0,T];\BS{L}^{2}(\Omega))\cap\mathcal{C}^{\theta}_{\sharp}([0,T];\BS{H}^{2}(\Omega))\right)\times \mathcal{C}^{\theta}_{\sharp}([0,T];H^{1}(\Omega))
\end{equation}
obeys the fluid equations of (\ref{chap2-linear-existence2}) if and only if
\begin{equation}
\begin{aligned}
&\Pi\vg_{t} = A_{s}\Pi\vg - A_{s}\Pi L_{1}(\eta_{t}),\,\,\,\vg(0)=\vg(T),\\
&(I-\Pi)\vg=\nabla N_{s}(\eta_{t}),\\
&q=-N_{s}(\eta_{t})_{t}+ N_{v}(\Pi\vg)+N_{p}(F).\\
\end{aligned}
\end{equation}
\end{theorem}
\begin{proof}
A pair $(\vg,q)$ as in \eqref{chap2-reg.linper} is solution to the fluid equations in \eqref{chap2-linear-existence2} if and only if
\[\begin{aligned}
&-\nu\Delta\vg+\nabla q=F-\vg_{t},\,\,\,\DIV\vg=0\text{ in }Q_{T},\\
&\vg=\eta_{t}\textbf{e}_{2}\,\text{ on }\Sigma^{s}_{T},\,\,\, \vg=0\,\text{ on }\Sigma^{i}_{T},\\
&v_{2}=0\,\text{ and }q=0\text{ on }\Sigma^{o}_{T},\,\,\,\vg=0\,\text{ on }\Sigma^{b}_{T}.
\end{aligned}
\]
We then apply Theorem \ref{chap2-reformulation.stokes} to conclude.
\end{proof}
Introduce the space
\[\BS{H}=\BS{V}^{0}_{n,\Gamma_{d}}(\Omega)\times H^{2}_{0}(\Gamma_{s})\times L^{2}(\Gamma_{s}),\]
equipped with the inner product
\[\langle (\ug,\eta_1,\eta_2),(\vg,\zeta_1,\zeta_2)\rangle_{\BS{H}}=\langle \ug,\vg\rangle_{\BS{L}^{2}(\Omega)}+\langle \eta_1,\zeta_1\rangle_{H^{2}_{0}(\Gamma_s)}+\langle \eta_2,\zeta_2 \rangle_{L^{2}(\Gamma_s)}.\]
Owing to Theorem \ref{chap2-lin.recast}, System (\ref{chap2-linear-existence2}) can be recast in terms of $(\Pi\vg,\eta,\eta_t)$:
\begin{equation}\label{chap2-evolution-existence1}
\begin{cases}
\begin{aligned}
&\frac{d}{dt}\begin{pmatrix}
\Pi\vg\\
\eta\\
\eta_t\\
\end{pmatrix}
=\mathcal{A}
\begin{pmatrix}
\Pi\vg\\
\eta\\
\eta_t\\
\end{pmatrix}
+\BS{f},\,\,\,\begin{pmatrix}
\Pi\vg(0)\\
\eta(0)\\
\eta_{t}(0)\\
\end{pmatrix}=\begin{pmatrix}
\Pi\vg(T)\\
\eta(T)\\
\eta_{t}(T)\\
\end{pmatrix},\\
&(I-\Pi)\vg=\nabla N_{s}(\eta_{t}),\\
&q=-N_{s}(\eta_{t})_{t}+ N_{v}(\Pi\vg)+N_{p}(F),\\
\end{aligned}
\end{cases}
\end{equation}
where
\[\BS{f}=
\begin{pmatrix}
\Pi F\\
0\\
(I+N_s)^{-1}(N_{p}(F)+H)\\
\end{pmatrix},
\]
and $\mathcal{A}$ is the unbounded operator in $\BS{H}$ defined by
\[
\begin{aligned}
\mathcal{D}(\mathcal{A})=\{(\Pi\vg,\eta_1,\eta_2)\in \BS{V}^{2}_{n,\Gamma_{d}}(\Omega)\times{}& (H^{4}(\Gamma_s)\cap H^{2}_{0}(\Gamma_s))\times H^{2}_{0}(\Gamma_s)\\
&\big| \Pi\vg -\Pi L_{1}(\eta_2)\in \mathcal{D}(A_{s})\},
\end{aligned}
\]
and
\begin{equation}\label{chap2-s5matriceA}
\mathcal{A}=
\begin{pmatrix}
I&0&0\\
0&I&0\\
0&0&(I+N_s)^{-1}\\
\end{pmatrix}
\begin{pmatrix}
A_{s}&0&-A_{s}\Pi L_{1}\\
0&0&I\\
N_{v}&A_{\alpha,\beta}&\delta\Delta_s\\
\end{pmatrix},
\end{equation}
with $\Delta_s=\partial_{xx}$.

\subsection{Analyticity of $\mathcal{A}$}
The unbounded operator $\mathcal{A}$ has already been studied, with small variations related to the boundary conditions, in \cite{MR2745779,2017arXiv170706382C}.
\begin{theorem}\label{chap2-s5thm1}The operator $(\mathcal{A},\mathcal{D}(\mathcal{A}))$ is the infinitesimal generator of an analytic semigroup on $\BS{H}$. Moreover, the resolvent of $(\mathcal{A},\mathcal{D}(\mathcal{A}))$ is compact.
\end{theorem}
\begin{proof} We write $\mathcal{A}=\mathcal{A}_{1}+\mathcal{A}_{2}$ with
\[\mathcal{A}_{1}=\begin{pmatrix}
A_{s}&0&-A_{s}\Pi L_{1}\\
0&0&I\\
0&A_{\alpha,\beta}&\delta\Delta_{s}\\
\end{pmatrix},
\]
\[\mathcal{A}_{2}=\begin{pmatrix}
0&0&0\\
0&0&0\\
(I+N_{s})^{-1}N_{v}&K_{s}A_{\alpha,\beta}&K_{s}\delta\Delta_{s}\\
\end{pmatrix},
\]
where $K_{s}=(I+N_{s})^{-1}-I$. We start with the resolvent of $\mathcal{A}_{1}$. Let $\lambda_{b}\in\mathbb{R}$ be such that $\{\lambda\in\mathbb{C}\mid \text{Re }\lambda\geq \lambda_{b}\}\subset\rho(\mathcal{A}_{b})$. For $\lambda\in\mathbb{C}$ such that $\text{Re }\lambda\geq \lambda_{b}$, consider the system
\begin{equation}\label{chap2-equation.A1}
\begin{aligned}
&\lambda\ug-\nu\Delta\ug + \nabla p = F_{1},\,\,\,\DIV\ug=0\,\text{ in }\Omega,\\
&\ug=\eta_{2}\textbf{e}_{2}\,\text{ on }\Gamma_{s},\,\,\,
\ug=0\,\text{ on }\Gamma_{i},\\
&u_2=0\text{ and }p=0\,\text{ on }\Gamma_{o},\\
&\ug=0\,\text{ on }\Gamma_{b},\\
&\lambda\eta_{1}-\eta_{2}=F_{2}\text{ on }\Gamma_{s},\\
&\lambda\eta_{2}-\beta\eta_{1,xx}-\gamma\eta_{2,xx}+\alpha\eta_{1,xxxx}=F_{3}\,\text{ on }\Gamma_{s},\\
&\eta_{1}=0\,\text{ and }\,\eta_{1,x}=0\,\text{ on }\{0,L\},\\
\end{aligned}
\end{equation}
for $(F_{1},F_{2},F_{3})\in\BS{H}$. This system is triangular: the beam equation can be solved first, and its solution injected in the Stokes system. The assumption on $\lambda$ ensures the existence of $(\eta_{1},\eta_{2})\in \left(H^{4}(\Gamma_{s})\cap H^{2}_{0}(\Gamma_{s})\right)\times H^{2}_{0}(\Gamma_{s})$ solution to the beam equations with the estimate,
\[\norme{\eta_{1}}_{H^{4}(\Gamma_{s})\cap H^{2}_{0}(\Gamma_{s})}+\norme{\eta_{2}}_{H^{2}_{0}(\Gamma_{s})}\leq C(\norme{F_{2}}_{H^{2}_{0}(\Gamma_{s})}+\norme{F_{3}}_{L^{2}(\Gamma_{s})}).\]
The Stokes system can then be solved, and we find $(\ug,p)\in \BS{H}^{2}(\Omega)\times H^{1}(\Omega)$ solution to $\eqref{chap2-equation.A1}_{1-4}$ such that
\begin{align*}
\norme{\ug}_{\BS{H}^{2}(\Omega)}+\norme{p}_{H^{1}(\Omega)}\leq{}& C(\norme{\eta_{2}}_{H^{2}_{0}(\Gamma_{s})}+\norme{F_{1}}_{\BS{L}_{2}(\Omega)})\\
\leq{}& C(\norme{F_{1}}_{\BS{L}_{2}(\Omega)}+\norme{F_{2}}_{H^{2}_{0}(\Gamma_{s})}+\norme{F_{3}}_{L^{2}(\Gamma_{s})}).
\end{align*}
System \eqref{chap2-equation.A1} is equivalent to
\[\begin{cases}
\begin{aligned}
&\lambda\begin{pmatrix}
\Pi\ug\\
\eta_{1}\\
\eta_{2}\\
\end{pmatrix}=\mathcal{A}_{1}\begin{pmatrix}
\Pi\ug\\
\eta_{1}\\
\eta_{2}
\end{pmatrix}+\begin{pmatrix}
F_{1}\\
F_{2}\\
F_{3}
\end{pmatrix},\\
&(I-\Pi)\ug=\nabla N_{s}(\eta_{2}),\\
&p=-\lambda N_{s}(\eta_{2})+N_{v}(\Pi\ug),
\end{aligned}
\end{cases}
\]
and the reasoning above shows that $\{\lambda\in\mathbb{C}\mid \text{Re }\lambda\geq \lambda_{b}\}\subset \rho(\mathcal{A}_{1})$. The resolvent estimates on $\mathcal{A}_{1}$ are similar to \cite[Theorem 3.2]{2017arXiv170706382C} and $(\mathcal{A}_{1},\mathcal{D}(\mathcal{A}_{1})=\mathcal{D}(\mathcal{A}))$ is sectorial. Using a similar technique as in \cite[Lemma 5.3]{2017arXiv170706382C} we prove that $(\mathcal{A}_{1},\mathcal{D}(\mathcal{A}_{1}))$ is the infinitesimal generator of a strongly continuous semigroup on $\BS{H}$. Finally, the previous properties imply that $(\mathcal{A}_{1},\mathcal{D}(\mathcal{A}_{1}))$ is the infinitesimal generator of an analytic semigroup on $\BS{H}$. 

As in \cite[Theorem 3.3]{2017arXiv170706382C}, the term $\mathcal{A}_{2}$ is $\mathcal{A}_{1}$-bounded with relative bound zero. Using \cite[Section 3.2, Theorem 2.1]{MR0512912}, we thus obtain the analyticity of $(\mathcal{A},\mathcal{D}(\mathcal{A}))$. The Rellich compact embedding theorem ensures that $\mathcal{D}(\mathcal{A})\overset{c}{\hookrightarrow} \BS{H}$ and the resolvent of $\mathcal{A}$ is therefore compact.
\end{proof}

\subsection{Time-periodic solutions of the linear system}\label{chap2-time-periodic-linear}
In this section we apply the existence results of periodic solutions developed in the appendix to the system \eqref{chap2-linear-existence1}.

In the appendix, we prove the existence of time-periodic solutions for abstract evolution equations $y'(t)=Ay(t)+f(t)$ under the assumption \eqref{chap2-assumptionAT}. This assumption is a restriction on the period $T$ of the system depending on the eigenvalues of $A$ lying on the imaginary axis. Here, this condition does not restrict the choice of $T$ as we are able to prove that all the non-zero eigenvalues of $\mathcal{A}$ have a negative real part. Indeed, let $\lambda\in\mathbb{C}$ be a non-zero eigenvalues of $\mathcal{A}$ and $(\Pi\ug,\eta_{1},\eta_{2})\in\mathcal{D}(\mathcal{A})$ be an associated eigenvector. The system
\[\lambda\begin{pmatrix}
\Pi\ug\\
\eta_{1}\\
\eta_{2}\\
\end{pmatrix}-\mathcal{A}\begin{pmatrix}
\Pi\ug\\
\eta_{1}\\
\eta_{2}\\
\end{pmatrix}=0,\]
is equivalent to
\begin{equation}\label{chap2-eigenvalue-neg}
\begin{aligned}
&\lambda\ug-\nu\Delta\ug + \nabla p = 0,\,\,\,\DIV\ug=0\,\text{ in }\Omega,\\
&\ug=\eta_{2}\textbf{e}_{2}\,\text{ on }\Gamma_{s},\,\,\,
\ug=0\,\text{ on }\Gamma_{i},\\
&u_2=0\text{ and }p=0\,\text{ on }\Gamma_{o},\\
&\ug=0\,\text{ on }\Gamma_{b},\\
&\lambda\eta_{1}-\eta_{2}=0\text{ on }\Gamma_{s},\\
&\lambda\eta_{2}-\beta\eta_{1,xx}-\gamma\eta_{2,xx}+\alpha\eta_{1,xxxx}=p\,\text{ on }\Gamma_{s},\\
&\eta_{1}=0\,\text{ and }\,\eta_{1,x}=0\,\text{ on }\{0,L\},\\
\end{aligned}
\end{equation}
with $\ug=\Pi\ug + \nabla N_{s}(\eta_{2})$ and $p=-\lambda N_{s}(\eta_{2})+N_{v}(\Pi\ug)$. Multiplying the first line of \eqref{chap2-eigenvalue-neg} by $\overline{\ug}$ (the complex conjugate of $\ug$) and integrating by part we obtain
\[\lambda\int_{\omega}\vert\ug\vert^{2} + \nu\int_{\Omega}\vert\nabla\ug\vert^{2}+\int_{\Gamma_{s}}p\overline{\eta}_{2}=0.\]
Then, multiplying the second line of the beam equation by $\overline{\eta}_{2}$, using the identity $\lambda\eta_{1}=\eta_{2}$ and integration by part we obtain
\[\int_{\Gamma_{s}}p\overline{\eta}_{2}=\lambda\int_{\Gamma_{s}}\vert\eta_{2}\vert^{2}+\beta\overline{\lambda}\int_{\Gamma_{s}}\vert\eta_{1,x}\vert^{2}+\gamma\int_{\Gamma_{s}}\vert\eta_{2,x}\vert^{2}+\alpha\overline{\lambda}\int_{\Gamma_{s}}\vert\eta_{1,xx}\vert^{2}.\]
Combining the previous energy estimates we obtain
\[\lambda\left[\int_{\Omega}\vert\ug\vert^{2}+\int_{\Gamma_{s}}\vert\eta_{2}\vert^{2} \right] + \overline{\lambda}\left[ \beta\int_{\Gamma_{s}}\vert\eta_{1,x}\vert^{2}+\alpha\int_{\Gamma_{s}}\vert\eta_{1,xx}\vert^{2} \right] + \nu\int_{\Omega}\vert\nabla\ug\vert^{2}+\gamma\int_{\Gamma_{s}}\vert\eta_{2,x}\vert^{2}=0.\]
Taking the real part of the previous identity we deduce that $\text{Re }\lambda<0$. It is easily verified that $0\not\in \sigma_{p}(\mathcal{A})$ (recall that $\sigma_{p}(\mathcal{A})=\sigma(\mathcal{A})$ as the resolvent of $\mathcal{A}$ is compact) and we can apply Theorem \ref{chap2-thm-C0} to solve the linear system \eqref{chap2-evolution-existence1} without restriction on the period $T$. Let $\BS{W}$ be the set defined by
\begin{align*}
\BS{W}:={}&\mathcal{C}^{\theta}_{\sharp}([0,T];\BS{L}^{2}(\Omega))\times \left(\mathcal{C}^{1+\theta}_{\sharp}([0,T];\BS{L}^{2}(\Omega))\cap \mathcal{C}^{\theta}_{\sharp}([0,T];\BS{H}^{2}(\Omega)\cap \BS{H}^{1}_{0}(\Omega))\right)\\
&\times \mathcal{C}^{\theta}_{\sharp}([0,T];H^{1/2}(\Gamma_{o}))\times \mathcal{C}^{\theta}_{\sharp}([0,T];L^{2}(\Gamma_{s})).
\end{align*}
The regularity space for the beam is denoted by 
\[
\mathcal{C}^{\theta}_{\text{beam}}:=\mathcal{C}^{\theta}_{\sharp}([0,T];H^{4}(\Gamma_{s})\cap H^{2}_{0}(\Gamma_{s}))\cap \mathcal{C}^{1+\theta}_{\sharp}([0,T];H^{2}_{0}(\Gamma_{s}))\cap \mathcal{C}^{2+\theta}_{\sharp}([0,T];L^{2}(\Gamma_{s})).\]
\begin{theorem}\label{chap2-mainthmlinear}For all $T>0$ and $(\BS{f},\BS{w},\Theta,h)\in \BS{W}$, (\ref{chap2-linear-existence1}) admits a unique periodic solution 
\[
(\ug,p,\eta)\in \left(\mathcal{C}^{1+\theta}_{\sharp}([0,T];\BS{L}^{2}(\Omega))\cap\mathcal{C}^{\theta}_{\sharp}([0,T];\BS{H}^{2}(\Omega))\right)\times\mathcal{C}^{\theta}_{\sharp}([0,T];H^{1}(\Omega))\times\mathcal{C}^{\theta}_{\text{beam}}.\]
Moreover $(\ug(0),\eta(0),\eta_{t}(0))$ is given by
\[\ug(0)=\Pi\vg(0)+\nabla N_{s}(\eta_{t}(0))+\BS{w}(0)+L_{\Gamma_{i},1}(\boldsymbol{\omega}_{1})(0)\,\text{ and }
\begin{pmatrix}
\Pi\vg(0)\\
\eta(0)\\
\eta_t(0)\\
\end{pmatrix}=P_{\mathcal{A}}\BS{f},
\]
where $P_{\mathcal{A}}$ is defined in Lemma \ref{chap2-lemmaAn2}. Finally, the following estimate holds
\begin{equation}\label{chap2-Estlineairemain}\begin{aligned}
&\norme{\ug}_{\mathcal{C}^{1+\theta}_{\sharp}([0,T];\BS{L}^{2}(\Omega))\cap\mathcal{C}^{\theta}_{\sharp}([0,T];\BS{H}^{2}(\Omega))}+\norme{p}_{\mathcal{C}^{\theta}_{\sharp}([0,T];H^{1}(\Omega))}+\norme{\eta}_{\mathcal{C}^{\theta}_{\text{beam}}}\\
&{}\leq C_{L}\Big[\norme{\boldsymbol{\omega}_{1}}_{\mathcal{C}^{\theta}_{\sharp}([0,T];\BS{H}^{3/2}_{0}(\Gamma_{i}))\cap \mathcal{C}^{1+\theta}_{\sharp}([0,T];\BS{H}^{-1/2}(\Gamma_{i}))}+\norme{\omega_{2}}_{\mathcal{C}^{\theta}_{\sharp}([0,T];H^{1/2}(\Gamma_{o}))}\\
&\qquad+\norme{(\BS{f},\BS{w},\Theta,h)}_{\BS{W}}\Big].
\end{aligned}
\end{equation}
\end{theorem}
\section{Nonlinear system}\label{chap2-nonlinear}
In this section we prove the existence of classical solutions for the nonlinear system (\ref{chap2-mainexistence}) using a fixed point argument. Without additional source terms in the model, here given through the inflow and outflow boundary conditions, the solution obtained with the fixed point procedure is the null solution. Hence, in what follows, the pair $(\boldsymbol{\omega}_{1},\omega_{2})$ is assumed to be non trivial, eventually small enough, and represent the `impulse' of the system. The period $T$ of $(\boldsymbol{\omega}_{1},\omega_{2})$ determines the period of the whole system.

Let $T>0$ be a fixed time and 
\[(\boldsymbol{\omega}_{1},\omega_{2})\in\left(\mathcal{C}^{\theta}_{\sharp}([0,T];\BS{H}^{3/2}_{0}(\Gamma_{i}))\cap \mathcal{C}^{1+\theta}_{\sharp}([0,T];\BS{H}^{-1/2}(\Gamma_{i}))\right)\times\mathcal{C}^{\theta}_{\sharp}([0,T];H^{1/2}(\Gamma_{o})).\]
Consider the Banach space $\mathcal{X}$ defined by
\[
\mathcal{X}=\left(\mathcal{C}^{1+\theta}_{\sharp}([0,T];\BS{L}^{2}(\Omega))\cap\mathcal{C}^{\theta}_{\sharp}([0,T];\BS{H}^{2}(\Omega))\right)\times\mathcal{C}^{\theta}_{\sharp}([0,T];H^{1}(\Omega))
\times \mathcal{C}^{\theta}_{\text{beam}},\\
\]
and
\begin{align*}
\mathcal{B}(R,\mu)=\{(\ug,p,\eta)\in\mathcal{X}&\,\vert\,\norme{(\ug,p,\eta)}_{\mathcal{X}}\leq R,\,\,\,\norme{(1+\eta)^{-1}}_{\mathcal{C}([0,T]\times\Gamma_{s})}\leq \mu\},
\end{align*}
with $R>0$ and $\mu>0$.
\begin{theorem}\label{chap2-estNLmain}Let $R>0$, $\mu>0$ and $(\ug,p,\eta)\in\mathcal{B}(R,\mu)$. There exists a polynomial $Q\in \mathbb{R}^{+}[X]$ satisfying $Q(0)=0$ and a constant $C(\mu)>0$ such that the following estimates hold
\[||\underbrace{(\BS{G}(\ug,p,\eta),\BS{w}(\ug,\eta),(1/2)\vert\ug\vert^{2},\Psi(\ug,\eta))}_{=:\BS{F}(\ug,p,\eta)}||_{\BS{W}}\leq C(\mu)Q(R)\norme{(\ug,p,\eta)}_{\mathcal{X}},\]
and for $(\ug_{i},p_{i},\eta_{i})\in \mathcal{B}(R,\mu)$ ($i=1,2$)
\begin{equation}\label{chap2-lipstchiz-main}
\norme{\BS{F}(\ug_{1},p_{1},\eta_{1})-\BS{F}(\ug_{2},p_{2},\eta_{2})}_{\BS{W}}\leq C(\mu)Q(R)\norme{(\ug_{1},p_{1},\eta_{1})-(\ug_{2},p_{2},\eta_{2})}_{\mathcal{X}}
\end{equation}
\end{theorem}
\begin{proof}
The nonlinear terms $(\BS{G}(\ug,p,\eta),\BS{w}(\ug,\eta),(1/2)\vert\ug\vert^{2},\Psi(\ug,p))$ are already estimated in \cite[Section 4.1]{2017arXiv170706382C} with explicit time dependency for Sobolev regularity in time. Here the time dependency is straightforward as all the functions involved in the estimates are H\"older continuous and $T$ is fixed. For example:
\[
\norme{\eta\ug_{t}}_{\mathcal{C}^{\theta}([0,T];\BS{L}^{2}(\Omega))}=\norme{\eta\ug_{t}}_{\mathcal{C}([0,T];\BS{L}^{2}(\Omega))}+\sup_{t_{1}\neq t_{2}}\frac{\norme{\eta(t_{1})\ug_{t}(t_{1})-\eta(t_{2})\ug_{t}(t_{2})}_{\BS{L}^{2}(\Omega)}}{\vert t_{1}-t_{2}\vert^{\theta}},\]
and the following estimates hold
\[\norme{\eta\ug_{t}}_{\mathcal{C}([0,T];\BS{L}^{2}(\Omega))}\leq \norme{\eta}_{\mathcal{C}([0,T];L^{\infty}(\Gamma_{s}))}\norme{\ug_{t}}_{\mathcal{C}([0,T];\BS{L}^{2}(\Omega))}\]
\begin{align*}
&\sup_{t_{1}\neq t_{2}}\frac{\norme{\eta(t_{1})\ug_{t}(t_{1})-\eta(t_{2})\ug_{t}(t_{2})}_{\BS{L}^{2}(\Omega)}}{\vert t_{1}-t_{2}\vert^{\theta}}\\
&\leq \sup_{t_{1}\neq t_{2}}\frac{\norme{\eta(t_{1})-\eta(t_{2})}_{L^{\infty}(\Gamma_{s})}}{\vert t_{1}-t_{2}\vert^{\theta}}\norme{\ug_{t}}_{\mathcal{C}([0,T];\BS{L}^{2}(\Omega))}\\
&+\sup_{t_{1}\neq t_{2}}\frac{\norme{\ug_{t}(t_{1})-\ug_{t}(t_{2})}_{\BS{L}^{2}(\Omega)}}{\vert t_{1}-t_{2}\vert^{\theta}}\norme{\eta}_{\mathcal{C}([0,T];L^{\infty}(\Gamma_{s}))}\\
&\leq \norme{\eta}_{\mathcal{C}^{\theta}([0,T];H^{4}(\Gamma_{s}))}\norme{\ug_{t}}_{\mathcal{C}([0,T];\BS{L}^{2}(\Omega))}+\norme{\ug_{t}}_{\mathcal{C}^{\theta}([0,T];\BS{L}^{2}(\Omega))}\norme{\eta}_{\mathcal{C}([0,T];L^{\infty}(\Gamma_{s}))}\\
&\leq 2\norme{\eta}_{\mathcal{C}^{\theta}([0,T];H^{4}(\Gamma_{s}))}\norme{\ug_{t}}_{\mathcal{C}^{\theta}([0,T];\BS{L}^{2}(\Omega))}.
\end{align*}
The other `ball' estimates and the Lipschitz estimates \eqref{chap2-lipstchiz-main} are obtained through the same techniques using the following Sobolev embeddings
\begin{align*}
&\norme{\eta}_{\mathcal{C}^{\theta}([0,T];L^{\infty}(\Gamma_{s}))}+\norme{\eta_{x}}_{\mathcal{C}^{\theta}([0,T];L^{\infty}(\Gamma_{s}))}+\norme{\eta_{xx}}_{\mathcal{C}^{\theta}([0,T];L^{\infty}(\Gamma_{s}))}+\norme{\eta_{xxx}}_{\mathcal{C}^{\theta}([0,T];L^{\infty}(\Gamma_{s}))}\\
&+\norme{\eta_{t}}_{\mathcal{C}^{\theta}([0,T];L^{\infty}(\Gamma_{s}))}+\norme{\eta_{tx}}_{\mathcal{C}^{\theta}([0,T];L^{\infty}(\Gamma_{s}))}\leq C\norme{\eta}_{\mathcal{C}^{\theta}([0,T];H^{4}(\Gamma_{s}))\cap \mathcal{C}^{1+\theta}([0,T];H^{2}(\Gamma_{s}))}.
\end{align*}
Finally remarks that all the nonlinear terms are at least quadratic and thus are bounded by $\norme{(\ug,p,\eta)}_{\mathcal{X}}^{\alpha}$ for $\alpha\geq 2$. Writing $\norme{(\ug,p,\eta)}_{\mathcal{X}}^{\alpha}\leq R^{\alpha-1}\norme{(\ug,p,\eta)}_{\mathcal{X}}$, with $\alpha-1\geq 1$, concludes the proof.
\end{proof}
For $R>0$ and $\mu>0$ introduce the map
\begin{equation}\label{chap2-T}
\mathcal{F} :
\begin{cases}
\mathcal{B}(R,\mu)&\longrightarrow \mathcal{X},\\
(\ug,p,\eta)&\longmapsto (\ug^{*},p^{*},\eta^{*}),\\
\end{cases}
\end{equation}
where $(\ug^{*},p^{*},\eta^{*})$ is the solution to (\ref{chap2-linear-existence1}) with right-hand side
\[(\BS{f},\BS{w},\Theta,h)=(\BS{G}(\ug,p,\eta),\BS{w}(\ug,\eta),(1/2)\vert\ug\vert^{2},\Psi(\ug,\eta)).\]
\begin{theorem}\label{chap2-Fixpoint}There exists $R^{*}>0$ and $\mu^{*}>0$ such that for all 
\[(\boldsymbol{\omega}_{1},\omega_{2})\in\left(\mathcal{C}^{\theta}_{\sharp}([0,T];\BS{H}^{3/2}_{0}(\Gamma_{i}))\cap \mathcal{C}^{1+\theta}_{\sharp}([0,T];\BS{H}^{-1/2}(\Gamma_{i}))\right)\times\mathcal{C}^{\theta}_{\sharp}([0,T];H^{1/2}(\Gamma_{o})),\]
satisfying
\[\norme{\boldsymbol{\omega}_{1}}_{\mathcal{C}^{\theta}_{\sharp}([0,T];\BS{H}^{3/2}_{0}(\Gamma_{i}))\cap \mathcal{C}^{1+\theta}_{\sharp}([0,T];\BS{H}^{-1/2}(\Gamma_{i}))}+\norme{\omega_{2}}_{\mathcal{C}^{\theta}_{\sharp}([0,T];H^{1/2}(\Gamma_{o}))}\leq \frac{R^{*}}{2C_{L}},\]
where $C_{L}$ is the constant involved in (\ref{chap2-Estlineairemain}), system (\ref{chap2-mainexistence}) admits a unique solution $(\ug,p,\eta)$ in the ball $\mathcal{B}(R^{*},\mu^{*})$.
\end{theorem}
\begin{proof}
Let $R_{1}>0$ and $\mu^{*}>1$. In order to ensure that the map $\mathcal{F}$ is well defined from $\mathcal{B}(R^{*},\mu^{*})$ into itself (with $R^{*}$ to be defined) we control the estimate on $\norme{(1+\eta)^{-1}}_{\mathcal{C}([0,T]\times\Gamma_{s})}$ with the parameter $R_{1}$. Precisely, for all $(\ug,p,\eta)\in\mathcal{B}(R_{1},\mu^{*})$, the following estimate holds
\[\norme{\eta}_{\mathcal{C}([0,T]\times\Gamma_{s},\Gamma_{s})}\leq C_{\infty}R_{1},\]
with $C_{\infty}>0$ a positive constant. Then we choose $R_{2}<\frac{\mu^{*}-1}{C_{\infty}\mu^{*}}$ and for all $(\ug,p,\eta)\in\mathcal{B}(R_{2},\mu^{*})$ the following estimate holds
\[\norme{(1+\eta)^{-1}}_{\mathcal{C}([0,T]\times\Gamma_{s},\Gamma_{s})}\leq \frac{1}{1-C_{\infty}R_{2}}<\mu^{*}.\]
The linear estimate \ref{chap2-Estlineairemain} implies that, for all $(\ug,p,\eta)\in\mathcal{B}(R_{2},\mu^{*})$,
\begin{align*}
&\norme{\mathcal{F}(\ug,p,\eta)}_{\mathcal{X}}\leq C_{L}(\norme{\boldsymbol{\omega}_{1}}_{\mathcal{C}^{\theta}_{\sharp}([0,T];\BS{H}^{3/2}_{0}(\Gamma_{i}))\cap \mathcal{C}^{1+\theta}_{\sharp}([0,T];\BS{H}^{-1/2}(\Gamma_{i}))}+\norme{\omega_{2}}_{\mathcal{C}^{\theta}_{\sharp}([0,T];H^{1/2}(\Gamma_{o}))}\\
&+C(\mu^{*})Q(R_{2})\norme{(\ug,p,\eta)}_{\mathcal{X}}).
\end{align*}
We choose $0<R^{*}<R_{2}$ such that $C(\mu^{*})Q(R^{*})<\min(\frac{1}{2C_{L}},\frac{1}{2})$. Finally choose $(\boldsymbol{\omega}_{1},\omega_{2})$ such that
\[\norme{\boldsymbol{\omega}_{1}}_{\mathcal{C}^{\theta}_{\sharp}([0,T];\BS{H}^{3/2}_{0}(\Gamma_{i}))\cap \mathcal{C}^{1+\theta}_{\sharp}([0,T];\BS{H}^{-1/2}(\Gamma_{i}))}+\norme{\omega_{2}}_{\mathcal{C}^{\theta}_{\sharp}([0,T];H^{1/2}(\Gamma_{o}))}\leq \frac{R^{*}}{2C_{L}}.\]
At this point we proved that $\mathcal{F}$ is well defined from $\mathcal{B}(R^{*},\mu^{*})$ into itself. Moreover, using \eqref{chap2-lipstchiz-main}, the Lipschitz estimate
\[\begin{aligned}\norme{\mathcal{F}(\ug_{1},p_{1},\eta_{1})-\mathcal{F}(\ug_{2},p_{2},\eta_{2})}_{\mathcal{X}}&{}\leq C_{L}C(\mu^{*})Q(R^{*})\norme{(\ug_{1},p_{1},\eta_{1})-(\ug_{2},p_{2},\eta_{2})}_{\mathcal{X}}\\
&{}\leq \frac{1}{2}\norme{(\ug_{1},p_{1},\eta_{1})-(\ug_{2},p_{2},\eta_{2})}_{\mathcal{X}}.\end{aligned} \]
for all $(\ug_{i},p_{i},\eta_ {i})\in\mathcal{B}(R^{*},\mu)$ ($i=1,2$) shows that $\mathcal{F}$ is a contraction from $\mathcal{B}(R^{*},\mu^{*})$ into itself. The Banach fixed point theorem then ensures the existence of a solution to (\ref{chap2-mainexistence}).
\end{proof}
\begin{remark}
Notice that all the previous work can be done similarly with data 
\[(\boldsymbol{\omega}_{1},\omega_{2})\in\left(L^{2}(0,T;\BS{H}^{3/2}_{0}(\Gamma_{i}))\cap H^{1}_{\sharp}(0,T;\BS{H}^{-1/2}(\Gamma_{i}))\right)\times L^{2}(0,T;H^{1/2}(\Gamma_{o}).\]
Indeed, using Theorem \ref{chap2-theorem-L2}, the existence of a solution for the linear system is similar and the nonlinear estimates are provided in \cite[Section 4.1]{2017arXiv170706382C}. We obtained a solution
\[(\ug,p,\eta)\in \BS{H}^{2,1}_{\sharp}(Q_{T})\times L^{2}(0,T;H^{1}(\Omega))\times H^{4,2}_{\sharp}(\Sigma^{s}_{T}).\]
This proof of existence also applies to other boundary conditions. For instance, as soon as the Stokes problem admits a solution in $\BS{H}^{2}(\Omega)$ (e.g. for pressure boundary conditions on the inflow and the outflow, Dirichlet boundary condition, periodic boundary conditions...) the results are valid.
\end{remark}
\appendix
\section{Appendix: Abstract results on periodic evolution equations}\label{chap2-abstract.results}
Let $H$ be a Hilbert space (with norm $\norme{\cdot}$) and $A$ be the infinitesimal generator of an analytic semigroup $S(t)$ on $H$ with domain $\mathcal{D}(A)$. In this section we are interested in the existence of a $T$-periodic solution to the following abstract evolution equation
\begin{equation}\label{chap2-Absmain}
y'(t)=Ay(t)+f(t)\,\text{, for $t\in\mathbb{R}$},
\end{equation}
where $f:\mathbb{R}\rightarrow H$ is a $T$-periodic source term with a regularity to be specified. A $T$-periodic function $y$ is solution to \eqref{chap2-Absmain} if and only if its restriction to $[0,T]$ is solution to
\begin{equation}\label{chap2-Absmain1}
\begin{cases}
\begin{aligned}
&y'(t)=Ay(t)+f(t)\,\text{, for all }t\in[0,T],\\
&y(0)=y(T).
\end{aligned}
\end{cases}
\end{equation}
In this section, two frameworks are considered to study \eqref{chap2-Absmain1}. The Hilbert case, when $f\in L^{2}(0,T;H)$, and the continuous case when $f\in \mathcal{C}([0,T];H)$ (or $f$ is H\"older continuous). The Hilbert case provides powerful tools to study \eqref{chap2-Absmain1} through the existence of isomorphism theorems \cite[Theorem 3.1, part II, section 1.3]{MR2273323}. This framework is used to prove the existence of a unique solution to \eqref{chap2-Absmain1} under additional hypothesis on the operator $A$. The previous strategy is developed in Section \ref{chap1-hilbert.case}. When $f$ is continuous or H\"older continuous, we use the continuous theory for evolution equations to improve the regularity of this solution. In both case we are interested in the existence of strict solutions. For $y^{0}\in H$ and $f\in L^{2}(0,T;H)$ consider the evolution equation
\begin{equation}\label{chap2-def-sol-evol}
\begin{cases}
\begin{aligned}
&y'(t)=Ay(t)+f(t)\,\text{, for all }t\in[0,T],\\
&y(0)=y^{0}.
\end{aligned}
\end{cases}
\end{equation}
\begin{definition}\label{chap2-defi.sol}$\,$

\textit{(i)} $y$ is a strict solution of \eqref{chap2-def-sol-evol} in $L^{2}(0,T;H)$ if $y$ belongs to $L^{2}(0,T;\mathcal{D}(A))\cap H^{1}(0,T;H)$, $y'(t)=Ay(t)+f(t)\,\text{for a.e. }t\in[0,T]$, and $y(0)=y^{0}$.

\textit{(ii)} $y$ is a strict solution of \eqref{chap2-def-sol-evol} in $\mathcal{C}([0,T];H)$ if $y$ belongs to $\mathcal{C}([0,T];\mathcal{D}(A))\cap \mathcal{C}^{1}([0,T];H)$, $y'(t)=Ay(t)+f(t)\,\text{ for all }t\in[0,T]$, and $y(0)=y^{0}$.

\textit{(iii)} $y$ is a classical solution of \eqref{chap2-def-sol-evol} in $\mathcal{C}([0,T];H)$ if $y$ belongs to $\mathcal{C}((0,T];\mathcal{D}(A))\cap \mathcal{C}^{1}((0,T];H)\cap \mathcal{C}([0,T];H)$, $y'(t)=Ay(t)+f(t)\,\text{ for all }t\in[0,T]$, and $y(0)=y^{0}$.

\textit{(iv)} The function 
\[y(t)=S(t)y^{0}+\int_{0}^{t}S(t-s)f(s)ds,\]
is called the mild solution of problem \eqref{chap2-def-sol-evol} if $y$ belongs to $\mathcal{C}([0,T];H)$.
\end{definition}
In what follows we assume that the pair $(A,T)$ satisfies the assumption:
\begin{equation}\label{chap2-assumptionAT}
\begin{aligned}
&\textit{The resolvent of }A\textit{ is compact, }0\not\in \sigma_{p}(A)\textit{ and }T\in \mathbb{R}^{+}\setminus \{\frac{2k\pi}{b_{j}}\mid k\in\mathbb{Z},\,0\leq j\leq N_{A}\}\\
&\textit{where }\{ib_{j}\}_{0\leq j\leq N_{A}}\textit{ denote the non zero eigenvalues of }A\textit{ on the imaginary axis }i\mathbb{R}\\
&\textit{with }N_{A}\in\mathbb{N}\textit{ and }b_{j}\in\mathbb{R}^{*}\textit{ with }0\leq j\leq N_{A}.
\end{aligned}
\end{equation}
Remark that the assumptions $A$ generates an analytic semigroup and has a compact resolvent directly imply that $N_{A}$ is a finite number.
\subsection{Hilbert case}\label{chap1-hilbert.case}
In this section we obtain a simple criteria to ensure that the problem \eqref{chap2-Absmain1} admits a unique strict solution in $L^{2}(0,T;H)$.
\begin{lemma}\label{chap2-lemmaAn1}The evolution equation (\ref{chap2-Absmain1}) admits a strict solution in $L^{2}(0,T;H)$ if and only if the equation
\begin{equation}\label{chap2-equation-cini}
(I-S(T))z=\int_{0}^{T}S(T-s)f(s)ds.
\end{equation}
admits at least one solution $z\in [\mathcal{D}(A),H]_{1/2}$.
\end{lemma}
\begin{proof}
Suppose that (\ref{chap2-Absmain1}) admits a strict solution $y\in L^{2}(0,T;\mathcal{D}(A))\cap H^{1}(0,T;H)$. We recall, see \cite{MR0350177}, that $L^{2}(0,T;\mathcal{D}(A))\cap H^{1}(0,T;H)\subset\mathcal{C}([0,T];[\mathcal{D}(A),H]_{1/2})$. As this solution coincides with the mild solution given by the Duhamel formula we have
\[y(0)=y(T)=S(T)y(0)+\int_{0}^{T}S(T-s)f(s)ds,\]
and thus $z=y(0)$ satisfies (\ref{chap2-equation-cini}). Reciprocally if $z\in [\mathcal{D}(A),H]_{1/2}$ satisfies the equation (\ref{chap2-equation-cini}) then consider the evolution equation
\begin{equation}\label{chap2-eqlemme1}
\begin{cases}
\begin{aligned}
&y'(t)=Ay(t)+f(t)\,\text{, for all }t\in[0,T],\\
&y(0)=z.
\end{aligned}
\end{cases}
\end{equation}
The isomorphism theorem \cite[Theorem 3.1, part II, section 1.3]{MR2273323} shows that (\ref{chap2-eqlemme1}) admits a unique solution $y\in L^{2}(0,T;\mathcal{D}(A))\cap H^{1}(0,T;H)$. Finally this solution satisfies (\ref{chap2-Absmain1}) by choice of $z$.
\end{proof}
\begin{lemma}\label{chap2-lemmaAn2}Suppose that the pair $(A,T)$ satisfies the assumption \eqref{chap2-assumptionAT}. Then the equation (\ref{chap2-equation-cini}) admits a unique solution $z\in [\mathcal{D}(A),H]_{1/2}$. Moreover the operator $P_{A}$ defined by
\[P_{A}f=(I-S(T))^{-1}\int_{0}^{T}S(T-s)f(s)ds,\]
is a bounded linear operator from $L^{2}(0,T;H)$ into $[\mathcal{D}(A),H]_{1/2}$.
\end{lemma}
\begin{proof}
Let $v$ be the function defined by
\[v(t)=\int_{0}^{t}S(t-s)f(s)ds,\]
and remark that $v(t)\in [\mathcal{D}(A),H]_{1/2}$ for all $t\in [0,T]$. The analyticity of the semigroup $S(t)$ implies that $S(T)z\in \mathcal{D}(A^{n})$ for all $n\geq 0$ and $z\in H$. Hence a solution $z$ to (\ref{chap2-equation-cini}) has the same regularity to $v(T)$ i.e. is in $[\mathcal{D}(A),H]_{1/2}$.

The assumption that $A$ has a compact resolvent implies (see \cite[Theorem 3.3]{MR0512912} and recall that $S(t)$ is analytic and thus differentiable, which implies the continuity for the uniform operator topology for $t>0$) that $S(t)$ is a compact semigroup. Hence $\sigma(S(T))=\sigma_{p}(S(T))$ and the spectral mapping theorem $e^{T\sigma_{p}(A)}= \sigma_{p}(S(T))$, coupled with the assumption \eqref{chap2-assumptionAT}, shows that $1\in \rho(S(T))$. Thus $(I-S(T))z=w$ for $w\in [\mathcal{D}(A),H]_{1/2}\subset H$ can be rewritten $z=(I-S(T))^{-1}w\in H$ and this $z$ belongs to $[\mathcal{D}(A),H]_{1/2}$. We have proved that the operator $(I-S(T))$ is a bijection from $[\mathcal{D}(A),H]_{1/2}$ into itself. By definition $S(T)\in \mathcal{L}(H)$. Moreover, using the graph norm on $\mathcal{D}(A)$ and a classical estimate for analytic semigroups, we have for all $u\in\mathcal{D}(A)$
\[\norme{S(T)u}_{\mathcal{D}(A)}=\norme{S(T)u}_{H}+\norme{AS(T)u}_{H}\leq \norme{S(T)}_{\mathcal{L}(H)}\norme{u}_{H}+\frac{C}{T}\norme{u}_{H}\leq C\norme{u}_{\mathcal{D}(A)}.\]
Hence $(I-S(T))\in\mathcal{L}(\mathcal{D}(A))$ and by interpolation $(I-S(T))\in\mathcal{L}([\mathcal{D}(A),H]_{1/2})$. Finally the bounded inverse theorem implies that $(I-S(T))^{-1}$ is a bounded linear operator on $[\mathcal{D}(A),H]_{1/2}$. From the continuous embedding $L^{2}(0,T;\mathcal{D}(A))\cap H^{1}(0,T;H)\subset \mathcal{C}^{0}([0,T];[\mathcal{D}(A),H]_{1/2})$ we obtain that
\[\norme{v(T)}_{[\mathcal{D}(A),H]_{1/2}}\leq C(\norme{v}_{L^{2}(0,T;\mathcal{D}(A))}+\norme{v}_{H^{1}(0,T;H)})\leq C\norme{f}_{L^{2}(0,T;H)},\]
and
\[\norme{P_{A}f}_{[\mathcal{D}(A),H]_{1/2}}\leq C\norme{(I-S(T))^{-1}}_{\mathcal{L}([\mathcal{D}(A),H]_{1/2})}\norme{f}_{L^{2}(0,T;H)}.\]
\end{proof}
Hence we have proved the following theorem.
\begin{theorem}\label{chap2-theorem-L2}Suppose that the pair $(A,T)$ satisfies the assumption \eqref{chap2-assumptionAT}. Then the periodic evolution equation \eqref{chap2-Absmain1} admits a unique strict solution $y\in L^{2}(0,T;\mathcal{D}(A))\cap H^{1}_{\sharp}(0,T;H)$ in $L^{2}(0,T;H)$. The following estimate holds
\[\norme{y}_{L^{2}(0,T;\mathcal{D}(A))\cap H^{1}_{\sharp}(0,T;H)}\leq C\norme{f}_{L^{2}(0,T;H)}.\]
\end{theorem}
\begin{proof}
It remains to prove the estimate. Using \cite[Theorem 3.1, part II, section 1.3]{MR2273323} and Lemma \ref{chap2-lemmaAn2} we obtain
\begin{align*}
\norme{y}_{L^{2}(0,T;\mathcal{D}(A))\cap H^{1}_{\sharp}(0,T;H)}&\leq C(\norme{y^{0}}_{[\mathcal{D}(A),H]_{1/2}}+\norme{f}_{L^{2}(0,T;H)})\\
&\leq C(\norme{P_{A}f}_{[\mathcal{D}(A),H]_{1/2}}+\norme{f}_{L^{2}(0,T;H)})\\
&\leq C\norme{f}_{L^{2}(0,T;H)}.
\end{align*}
\end{proof}
Using the regularization properties of analytic semigroup for $t>0$, that is $S(t)z\in\mathcal{D}(A^{n})$ for all $n\geq 1$ and $z\in H$, we can prove that the regularity of the solution solely depends on the source term $f$. Hence the previous result can be improved when $f$ is more regular. We introduce the space $\mathcal{H}^{r}=[\mathcal{D}(A^{n+1}),\mathcal{D}(A^{n})]_{1-\alpha}$ with $r=n+\alpha$, $n\geq 0$ an integer and $0\leq \alpha \leq 1$ a real number.
\begin{lemma}\label{chap2-lemmaAn3}Let $f$ be in $L^{2}(0,T;\mathcal{H}^{r-1})$ with $r>1$ and suppose that the pair $(A,T)$ satisfies the assumption \eqref{chap2-assumptionAT}. Then the unique strict solution $y$ in $L^{2}(0,T;H)$ belongs to $y\in L^{2}(0,T;\mathcal{H}^{r})\cap H^{1}_{\sharp}(0,T;\mathcal{H}^{r-1})$ and $y(0)\in [\mathcal{H}^{r},\mathcal{H}^{r-1}]_{1/2}$.
\end{lemma}
\begin{proof}We split (\ref{chap2-Absmain1}) in two parts
\[
\begin{cases}
\begin{aligned}
&y_{1}'(t)=Ay_{1}(t)+f(t)\,\text{, for all }t\in[0,T],\\
&y_{1}(0)=0,
\end{aligned}
\end{cases}
\]
and
\[
\begin{cases}
\begin{aligned}
&y_{2}'(t)=Ay_{2}(t)\,\text{, for all }t\in[0,T],\\
&y_{2}(0)=z.
\end{aligned}
\end{cases}
\]
Using the analyticity of $S$ we have $y_{2}(T)=S(T)z\in\mathcal{D}(A^{n})$ for all $n\geq 1$. On the other hand \cite[Theorem 2.2, part II, section 3.2.1]{MR2273323} (and the remark following the theorem on the extension of the isomorphism theorem) implies that $y_{1}\in L^{2}(0,T;H^{r})\cap H^{1}(0,T;H^{r-1})\subset \mathcal{C}^{0}([0,T];[H^{r},H^{r-1}]_{1/2})$. Hence $y(T)=y_{1}(T)+y_{2}(T)$ is in $[H^{r},H^{r-1}]_{1/2}$. Then we use the periodic condition $y(T)=y(0)$ and again \cite[Theorem 2.2, part II, section 3.2.1]{MR2273323} to obtain $y\in L^{2}(0,T;\mathcal{H}^{r})\cap H^{1}_{\sharp}(0,T;\mathcal{H}^{r-1})$.
\end{proof}
\subsection{Continuous case}
Let us recall the fundamental existence and regularity result (see \cite[Theorem 1.2.1, Section II]{MR1345385}):
\begin{theorem}\label{chap2-amann-theorem}Suppose that $f\in \mathcal{C}^{\rho}([0,T];H)$ with $\rho\in(0,1)$ and $y^{0}\in H$. Then the Cauchy problem \eqref{chap2-def-sol-evol} possesses a unique classical solution $y$ in $\mathcal{C}([0,T];H)$ and
\[y\in \mathcal{C}^{\rho}((0,T];\mathcal{D}(A))\cap \mathcal{C}^{\rho+1}((0,T];H),\]
with the estimate, for all $\varepsilon>0$,
\[\norme{y}_{\mathcal{C}^{\rho}([\varepsilon,T];\mathcal{D}(A))\cap \mathcal{C}^{\rho+1}([\varepsilon,T];H)}\leq C(\norme{y(\varepsilon)}_{\mathcal{D}(A)}+\norme{f}_{\mathcal{C}^{\rho}([0,T];H)}).\]
If $y^{0}\in \mathcal{D}(A)$ then the solution is strict.
\end{theorem}
\begin{proof}
The estimate can be obtained following the steps of the proof in \cite[Theorem 1.2.1, Section II]{MR1345385}, see in particular \cite[Theorem 2.5.6, Section III]{MR1345385}.
\end{proof}
We are now able to prove the existence of a strict periodic solution in $\mathcal{C}([0,T];H)$. Moreover, the previous H\"older regularity result and the periodicity show that the periodic solution possesses H\"older regularity up to $t=0$.
\begin{theorem}\label{chap2-thm-C0}Let $f\in \mathcal{C}^{\rho}([0,T];H)$ with $\rho\in(0,1)$ and suppose that the pair $(A,T)$ satisfies the assumption \eqref{chap2-assumptionAT}. Then the periodic evolution equation \eqref{chap2-Absmain1} admits a unique strict solution $y$ in $\mathcal{C}([0,T];H)$. Moreover
\[y\in \mathcal{C}^{\rho}([0,T];\mathcal{D}(A))\cap \mathcal{C}^{\rho+1}([0,T];H),\]
and the following estimate holds
\begin{equation}\label{chap2-estimateyf}\norme{y}_{\mathcal{C}^{\rho}([0,T];\mathcal{D}(A))\cap \mathcal{C}^{\rho+1}([0,T];H)}\leq C\norme{f}_{\mathcal{C}^{\rho}([0,T];H)}.
\end{equation}
\end{theorem}
\begin{proof}
We already know that there exists a strict solution in $L^{2}(0,T;H)$. Keeping the notations used in Lemma \ref{chap2-lemmaAn3}, we split $y=y_{1}+y_{2}$. For $y_{2}$ we still have $y_{2}(T)\in\mathcal{D}(A)$. Theorem \ref{chap2-amann-theorem} implies that $y_{1}\in \mathcal{C}^{\rho}((0,T];\mathcal{D}(A))\cap \mathcal{C}^{\rho+1}((0,T];H)$, thus $y_{1}(T)\in \mathcal{D}(A)$. Then the periodic condition $y(0)=y(T)$ implies that $y(0)\in\mathcal{D}(A)$ and Theorem \ref{chap2-amann-theorem} ensures the existence of a strict solution in $\mathcal{C}([0,T];H)$. Finally, considering the $T$-periodic extension $\hat{y}$ of $y$ on $[0,2T]$ the H\"older regularity result implies that $\hat{y}\in \mathcal{C}^{\rho}((0,2T];\mathcal{D}(A))\cap \mathcal{C}^{\rho+1}((0,2T];H)$. Hence $\hat{y}$ is H\"older in a neighbourhood of $T$, which implies that $y\in\mathcal{C}^{\rho}([0,T];\mathcal{D}(A))\cap \mathcal{C}^{\rho+1}([0,T];H)$. It remains to estimate $y$ with respect to $f$. Let us fix $\varepsilon=\frac{T}{2}$. We have
\[y(\varepsilon)=S(\varepsilon)y^{0} + \int_{0}^{\varepsilon}S(\varepsilon-s)f(s)ds.\]
The homogeneous part was already estimated in Lemma \ref{chap2-lemmaAn2}
\[\norme{S(\varepsilon)y^{0}}_{\mathcal{D}(A)}\leq C\norme{y^{0}}_{H}.\]
The integral part in Duhamel can be estimated as follows
\[\int_{0}^{\varepsilon}AS(\varepsilon-s)f(s)ds=\int_{0}^{\varepsilon}AS(\varepsilon-s)(f(s)-f(\varepsilon))ds + \int_{0}^{\varepsilon}AS(\varepsilon-s)f(\varepsilon)ds,\]
and
\[\norme{\int_{0}^{\varepsilon}AS(\varepsilon-s)f(\varepsilon)ds}_{H}=\norme{(S(\varepsilon)-I)f(\varepsilon)}_{H}\leq C\norme{f}_{\mathcal{C}^{\rho}([0,T];H)},\]
where we have used $\displaystyle\frac{d}{dt}S(t)=AS(t)$.
Finally
\[\norme{\int_{0}^{\varepsilon}AS(\varepsilon-s)(f(s)-f(\varepsilon))ds}_{H}\leq \int_{0}^{\varepsilon}\frac{C}{\vert \varepsilon-s\vert}\vert\varepsilon-s\vert^{\rho}\norme{f}_{\mathcal{C}^{\rho}([0,T];H)}ds\leq C\norme{f}_{\mathcal{C}^{\rho}([0,T];H)},\]
and $\displaystyle\norme{y(\varepsilon)}_{\mathcal{D}(A)}\leq C\norme{f}_{\mathcal{C}^{\rho}([0,T];H)}$. The estimate in Theorem \ref{chap2-amann-theorem} implies that 
\[\norme{\hat{y}}_{\mathcal{C}^{\rho}([\varepsilon,2T];\mathcal{D}(A))\cap \mathcal{C}^{\rho+1}([\varepsilon,2T];H)}\leq C\norme{\hat{f}}_{\mathcal{C}^{\rho}([0,2T];H)},\]
where $\hat{f}$ is the $T$-periodic extension of $f$ to $[0,2T]$. Then, taking the restriction to a period $T$, we obtain the estimate \eqref{chap2-estimateyf}.
\end{proof}

\bibliographystyle{plain}
\bibliography{biblio_1}
\end{document}